\documentclass[a4 paper, 12pt,reqno]{amsart}
\usepackage{a4,amsmath,amssymb,amscd,verbatim,bbm,graphicx,enumerate,blkarray}
\usepackage[english]{babel}

\usepackage{amsfonts,amsthm}
\usepackage[colorlinks,citecolor=cyan,linkcolor=magenta]{hyperref}
\usepackage{graphicx}
\usepackage{bmpsize}
\usepackage{MnSymbol}
\usepackage{color}
\usepackage{mathrsfs}
\usepackage{xcolor}

\usepackage{appendix}

\usepackage[nameinlink]{cleveref}

\newtheorem{theorem}{Theorem}
\newtheorem{proposition}[theorem]{Proposition}
\newtheorem{corollary}[theorem]{Corollary}
\newtheorem{lemma}[theorem]{Lemma}

\theoremstyle{definition}

\newtheorem{definition}[theorem]{Definition}

\newtheorem{remark}[theorem]{Remark}

\newtheorem{example}[theorem]{Example}

\newtheorem{question}[theorem]{Question}

\numberwithin{equation}{section}
\numberwithin{theorem}{section}

\newcommand{\G}{\Gamma}

\renewcommand{\k}{\kappa}

\renewcommand{\l}{\lambda}

\renewcommand{\r}{\rho}

\newcommand{\Dc}{{\mathcal D}}

\newcommand{\R}{{\mathbb R}}

\newcommand{\Z}{{\mathbb Z}}

\newcommand{\calA}{\mathcal{A}}

\newcommand{\calC}{\mathcal{C}}
\newcommand{\calD}{\mathcal{D}}

\newcommand{\calG}{\mathcal{G}}
\newcommand{\calH}{\mathcal{H}}

\newcommand{\calX}{\mathcal{X}}

\newcommand{\scrC}{\mathscr{C}}
\newcommand{\scrD}{\mathscr{D}}

\newcommand{\scrT}{\mathscr{T}}

\newcommand{\scrV}{\mathscr{V}}

\newcommand{\Hy}{\mathbb{H}}

\newcommand{\bbK}{\mathbb{K}}

\newcommand{\bbP}{\mathbb{P}}

\newcommand{\al}{\alpha}
\newcommand{\gam}{\gamma}

\newcommand{\del}{\delta}

\newcommand{\Del}{\Delta}
\newcommand{\ep}{\epsilon}
\newcommand{\thet}{\theta}

\newcommand{\lam}{\lambda}

\newcommand\SL{\operatorname{SL}}
\newcommand\Aut{\textnormal{Aut}}
\newcommand\Out{\textnormal{Out}}
\newcommand{\Curr}{\calC urr}

\newcommand{\bs}{\backslash}

\newcommand{\ra}{\rightarrow}

\newcommand{\cor}[1]{\left\{{#1}\right\}}

\def\({\left(}
\def\){\right)}

\def\l\{{\left\{}
\def\r\}{\right\}}
\def\wt{\widetilde}
\def\wh{\widehat}

\def\wbar{\overline}
\def\ov{\overline}
\def\id{{\rm id}}

\def\Cay{{\rm Cay}}
\def\Leb{{\rm Leb}}
\def\ev{{\rm ev\,}}
\def\1{{\bf 1}}

\def\gr{{\rm gr}}
\def\conj{{\bf conj}}

\def\ac{{\rm ac}}
\def\sing{{\rm sing}}

\def\H{{\mathbb H}}

\def\bS{{\bf S}}

\def\CAT{{\rm CAT}}

\def\Dil{{\rm Dil}}

\DeclareMathOperator{\diam}{diam}

\newcommand\numberthis{\addtocounter{equation}{1}\tag{\theequation}}

\date{\today}

\title[Manhattan geodesics and the boundary of metric structures]{Manhattan geodesics and the boundary of the space of metric structures on hyperbolic groups}

\author{\small{Stephen Cantrell and Eduardo Reyes}}

\begin{document}
\maketitle
\begin{abstract}
For any non-elementary hyperbolic group $\G$, we find an outer automorphism invariant geodesic bicombing for the space of pseudo metric structures on $\G$ equipped with a symmetrized version of the Thurston metric on Teichm\"uller space. We construct and study a boundary for this space and show that it contains many well-known pseudo metrics.
As corollaries we obtain 
results regarding continuous extensions of translation length functions to the space of geodesic currents and settle a conjecture of Bonahon in the negative.
\end{abstract}
\maketitle

\section{Introduction}
Let $\G$ be a non-elementary hyperbolic group, and let $\calD_\G$ denote the set of all left-invariant, hyperbolic pseudo metrics on $\G$ that are quasi-isometric to a word metric. For $d\in \calD_\G$ the \emph{stable translation length} function is defined as
\begin{equation*}
    \ell_d[x]:=\lim_{n\to \infty}{\frac{1}{n}d(o,x^n)} \hspace{2mm} \text{for } x\in \G,
\end{equation*}
where $o$ denotes the identity element and $[x]$ is the conjugacy class containing $x$.

Given a pair of pseudo metrics $d,d_*\in \calD_\G$, their \emph{dilations} are given by
\begin{equation*}
    \Dil(d,d_*):=\sup_{[x] \in  \conj'}\frac{\ell_d[x]}{\ell_{d_*}[x]} \ \ \text{ and } \ \ \Dil(d_*,d):=\sup_{[x] \in  \conj'}\frac{\ell_{d_*}[x]}{\ell_{d}[x]},
\end{equation*}
where $\conj'$ is the set of conjugacy classes of non-torsion elements in $\G$.

The pseudo metrics $d,d_\ast$ are quasi-isometric to each other via the identity map on $\G$. Our first result states that the optimal quasi-isometry constants for this map are given by the dilations, which refines \cite[Lem.~3.4]{oregon-reyes.ms}. The following is a special case of a more general result which appears as Theorem \ref{thm.Dildistancelike} in Section \ref{sec.optqiconstant}.
\begin{theorem}\label{thm.dilatta}
    For any $d,d_*\in \calD_\G$ there exists some $C\geq 0$ such that 
    \begin{equation}\label{eq.dilatta}
        \Dil(d,d_*)^{-1}(x|y)_{o,d}-C\leq (x|y)_{o,d_*}\leq \Dil(d_*,d)(x|y)_{o,d}+C
    \end{equation}
    for all $x,y\in \G$.
\end{theorem}
In the inequality above, $(\cdot|\cdot)_{o,d}$ and $(\cdot|\cdot)_{o,d_*}$ denote the corresponding Gromov products for $d$ and $d'$, see Section \ref{sec.hypspacesgroups}. Hyperbolicity of $d$ and $d_\ast$ is crucial to obtain the inequality \eqref{eq.dilatta}, and in general this is strictly stronger than the analogous inequality with the distance between two points.
Indeed, if $d\in \calD_\G$ and $d_\ast$ is an arbitrary pseudometric satisfying \eqref{eq.dilatta} then $d$ must be hyperbolic, see Remark \ref{rmk.hdlfhyp}. On the other hand, there are non-hyperbolic left-invariant metrics on $\G$ that are quasi-isometric to word metrics, see e.g.~\cite[Prop.~A.11]{bhm}.  

As an immediate consequence, we get another proof of (weak) marked length spectrum rigidity for pseudo metrics in $\calD_\G$ \cite[Thm.~4.1]{furman}: if $d,d_*\in \calD_\G$ satisfy $\ell_d[x]=\ell_{d_*}[x]$ for all $x\in \G$, then $|d-d_*|\leq C$ for some constant $C\geq 0$. We note that our more general theorem, Theorem \ref{thm.Dildistancelike}, can be applied to interesting examples of distance-like functions that are not necessarily metrics in $\Dc_\G$. In Section \ref{sec.Anosov} we apply our Theorem \ref{thm.Dildistancelike} to study growth rate constants associated to Anosov representations.

The inequality \eqref{eq.dilatta} might seem innocent at first, but it has interesting consequences when it comes to the understanding of the space $\calD_\G$. Below we present applications of Theorem \ref{thm.dilatta}.

\subsection{Geodesics in the space of metric structures}
 In \cite{oregon-reyes.ms}, the second author studied the topological and metric properties of the space of \emph{metric structures} on $\G$, denoted by $\scrD_\G$. This space is the quotient of $\calD_\G$ under the equivalence relation of rough similarity (see Section \ref{sec.hypspacesgroups}), and is endowed with the distance
\begin{equation}\label{eq.defDel}
\Del([d],[d_*]):=\log  \left(\Dil(d,d_*) \Dil(d_*,d) \right).
\end{equation}

Among other properties, it was proven that $(\scrD_\G,\Del)$ is unbounded and contractible \cite[Thm.~1.3]{oregon-reyes.ms}. Also, $\scrD_\G$ contains Teichm\"uller space $\scrT_\G$ when $\G$ is a surface group, and the Culler-Vogtmann Outer space $\scrC\scrV(\G)$ in case $\G$ is a free group.

As the first application of Theorem \ref{thm.dilatta}, we prove that the metric space $(\scrD_\G,\Del)$ is geodesic. Indeed, every pair of distinct metric structures lie in a \emph{bi-infinite} geodesic.

\begin{theorem}\label{thm.Mangeo} For any  pair $d, d_* \in \calD_\G$ such that $[d]\neq [d_*]\in \scrD_\G$, there exists a continuous, injective map $\rho_{\bullet}=\rho_\bullet^{d_*/d}: \mathbb{R} \rightarrow \scrD_{\Gamma}$ satisfying:
\begin{itemize}
\item[$i)$] $\rho_0=[d]$ and $\rho_{h(d_*)}=[d_*]$ where $h(d_\ast)$ is the exponential growth rate of $d_\ast$;
\item[$ii)$] $\Delta\left(\rho_{r}, \rho_{t}\right)=\Delta\left(\rho_{r}, \rho_{s}\right)+\Delta\left(\rho_{s}, \rho_{t}\right)$ for all $r<s<t$; and
\item[$iii)$] $\lim _{t \rightarrow \infty} \Delta\left(\rho_{t}, \rho_{h(d_*)}\right)=\lim _{t \rightarrow-\infty} \Delta\left(\rho_{t}, \rho_{0}\right)=\infty$.
\end{itemize}
\end{theorem}
In particular, we get another proof that $\scrD_\G$ is unbounded.
The result above seems surprising when we contrast it with the case of Outer space, which is not geodesic for the symmetrized Thurston metric \cite[Sec.~6]{franc.mart}.

The map $\rho_\bullet^{d_*/d}$ in Theorem \ref{thm.Mangeo} is constructed in such a way that for any $t\in \R$, any pseudo metric representing $\rho_t^{d_*/d}$ is roughly similar to  $td_*+\thet(t)d$, where $\thet=\thet_{d_*/d}$ is the parametrization of the \emph{Manhattan curve} for $d,d_*$ (see Section \ref{sec.manhattan} and Proposition \ref{prop.da}). This happens even when $t$ or $\thet(t)$ are negative. Intuitively the Manhattan curve $\theta$ can be thought of as a scaling factor that fixes exponential growth rate for linear combinations of $d$ and $d_\ast$. More precisely, for each $t \in \R$, $\theta(t)$ is the unique real number such that the sum $td_\ast + \theta(t) d$ has exponential growth rate $1$. Therefore $\theta$ helps us to  interpolate between metrics whilst keeping exponential growth rate constant.

For $\rho=[d]$ and $\rho_*=[d_*]$ as above, we can consider the arc-length reparametrization of $\rho_\bullet^{d_*/d}$, denoted by $\sigma_\bullet^{\rho_*/\rho}$, such that $\sigma_0^{\rho_*/\rho}=\rho$ and $\sigma_{\Del(\rho,\rho_*)}^{\rho_*/\rho}=\rho_*$. Such reparametrization is independent of the representatives $d$, and $d_*$, so we call it the \emph{Manhattan geodesic} of the pair $\rho,\rho_*$, see Definition \ref{def.Manhgeo1}.
In this way, we produce a geodesic bicombing on $\scrD_\G$ given by $(\rho,\rho_*)\mapsto \sigma_\bullet^{\rho_*/\rho}$, which inherits some of the good behavior of the Manhattan curves, as we will prove in Theorem \ref{thm.bicombing}.

\subsection{Boundary metric structures and the Manhattan boundary}
To understand the behavior at infinity of the Manhattan geodesics, we extend our set $\scrD_\G$ to allow rough similarity classes of pseudo metrics on $\G$ that are not necessarily quasi-isometric to a word metric. 
\begin{definition}\label{def.Manhattanboundarymetrics}
    Let $\ov\calD_\G$ be the set of all the left-invariant pseudo metrics $d$ on $\G$ such that its stable translation length function is non-constant and there are some $\lam>0$ and $d_0\in \calD_\G$ such that 
     \begin{equation}\label{eq.ineqdefboundarymetric}
        (x|y)_{o,d}\leq \lam (x|y)_{o,d_0}+\lam
    \end{equation}
    for all $x,y\in \G$. We also set $\partial_M \calD_\G:=\ov\calD_\G \bs \calD_\G$.
\end{definition}
By Lemma \ref{lem.RGforHDLF} in Section \ref{sec.optqiconstant} and Lemma \ref{lem.charGro=BBT} in Section \ref{sec:examples}, every pseudo metric in $\ov\calD_\G$ is roughly geodesic and hyperbolic. Since hyperbolicity is preserved under quasi-isometry among roughly geodesic metric spaces, we have that $\calD_\G \subset \ov\calD_\G$. Moreover, a pseudo metric in $\ov\calD_\G$ belongs to $\calD_\G$ if and only if it is quasi-isometric to a word metric.

\begin{definition}\label{def.Manhattanboundarystructures}
    The \emph{Manhattan boundary} of $\scrD_\G$ is $\partial_M\scrD_\G$, the quotient of $\partial_M \calD_\G$ under the equivalence relation of rough similarity. Its elements are called \emph{boundary metric structures}. The \emph{closure} of $\scrD_\G$ is $\ov\scrD_\G:=\scrD_\G \cup \partial_M \scrD_\G$.
\end{definition}
As we show in Section \ref{sec:manhattanboundary}, the Manhattan boundary is non-empty. Indeed, by Theorem \ref{thm.dilatta} we deduce that for any two pseudo metrics $d,d_*\in \calD_\G$ that are not roughly similar, there exist pseudo metrics $d_\infty, d_{-\infty}\in \partial_M\calD_\G$ which are roughly isometric to $\Dil(d,d_*)d_*-d$ and $\Dil(d_*,d)d-d_*$, respectively. In addition, if for each $t\in \R$ we consider $d_t\in \calD_\G$ that is roughly isometric to $td_*+\thet_{d_*/d}(t)d$, then we have
\begin{equation}\label{eq.limitMgeod}
\ell_{d_\infty}[x]=\lim_{t \to \infty}{\frac{1}{-\thet_{d_*/d}(t)}\ell_{d_t}[x]} \hspace{2mm}\text{ and }\hspace{2mm}\ell_{d_{-\infty}}[x]=\lim_{t \to -\infty}{\frac{1}{-t}\ell_{d_t}}[x]
\end{equation}
for every $x\in \G$, see Proposition \ref{prop.d+-inf}. The rough similarity classes $[d_{\pm \infty}]$ are independent of the representatives $d, d_*$ in $\rho=[d], \rho_*=[d_*]$, so \eqref{eq.limitMgeod} motivates the following definition.

\begin{definition}\label{def.limitsofManhattan}
If $\sigma=\sigma_\bullet^{\rho_*/\rho}$ is the Manhattan geodesic for the pair $\rho=[d],\rho_*=[d_*]$ with $\rho\neq \rho_*$, the \emph{limit at infinity} of $\sigma$ is the unique boundary metric structure $\sigma^{\rho_*/\rho}_\infty \in \partial_M\scrD_\G$ such that every pseudo metric representing $\sigma^{\rho_*/\rho}_\infty$ is roughly similar to $\Dil(d,d_*)d_*-d$. Analogously, the \emph{limit at negative infinity} of $\sigma$ is the unique boundary metric structure $\sigma^{\rho_*/\rho}_{-\infty}$ whose pseudo metric representatives are roughly similar to $\Dil(d_*,d)d-d_*$.
\end{definition}
We will see in Section \ref{sec:manhattanboundary} that every boundary metric structure is the limit at infinity of some Manhattan geodesic. Indeed, we prove in Theorem \ref{thm.Mboundaryvisible} that we can choose this geodesic containing any given metric structure in $\scrD_\G$. In a future project, we plan to address potential topologies on the spaces $\partial_M\scrD_\G$ and $\ov\scrD_\G$.

\subsection{Examples of boundary metric structures }
Many interesting and widely studied isometric actions on hyperbolic spaces induce pseudo metrics in $\calD_\G$, and the same holds for $\partial_M \calD_\G$. We recover pseudo metrics on $\G$ by restricting to the orbits of these actions. In Section \ref{sec:examples}, we prove that the following actions induce pseudo metrics in $\ov\calD_\G$, and hence metric structures in $\ov\scrD_\G$.
\begin{theorem}\label{thm.examples}
The following actions induce points in $\ov\scrD_\G$. 
\begin{enumerate}
    \item Natural actions on coned-off Cayley graphs for finite, symmetric generating sets, where we cone-off a finite number of quasi-convex subgroups of infinite index.
    \item Non-trivial Bass-Serre tree actions with quasi-convex edge stabilizers of infinite index. More generally, cocompact actions on $\CAT(0)$ cube complexes with quasi-convex hyperplane stabilizers and without global fixed points.
    \item Small actions on $\R$-trees, when $\G$ is a surface group or a free group.
\end{enumerate}
\end{theorem}
Our study of $\ov\scrD_\G$ therefore provides a unified approach to understanding metrics coming from various parts of coarse geometry. Also, by item (3) above we deduce that when $\G$ is a surface (resp.~free group), the Manhattan boundary is an extension of the Thurston (resp.~Culler-Vogtmann) boundary for Teichm\"uller (resp.~Outer) space, see Corollaries \ref{cor.smallfreegroup} and \ref{coro.boundTeichembeds}. One might ask to what extent items (1) and (2) of the theorem above can be generalized to arbitrary acylindrical actions, see Question \ref{question.acylindrical}.

In the case of surface groups, we can say something stronger since we can embed the space $\bbP\Curr(\G)$ of \emph{projective geodesic currents} into $\ov\scrD_\G$, see Corollary \ref{coro.boundTeichembeds}. This is done by analyzing the pseudo metric $d_\mu$ on $\Hy^2$ for a non-zero geodesic current $\mu$, defined by Burger, Iozzi, Parreau and Pozzetti \cite[Sec.~4]{burger-iozzi-parreau-pozzetti}. In a recent paper \cite{margra.dero}, Martinez-Granado and de Rosa study the pseudo metrics $d_{\mu}$ in more detail.

In the case of free groups, item (3) above follows since small actions of free groups on $\R$-trees have \emph{bounded backtracking}, which was proven by Guirardel \cite[Cor. 2]{guirardel}. Indeed, the inequality \eqref{eq.ineqdefboundarymetric} is a generalization of bounded backtracking for actions on hyperbolic spaces that are not necessarily trees, see Lemma \ref{lem.BCCchar}. In a forthcoming work \cite{kap.margra}, Kapovich and Martinez-Granado show that for freely indecomposable hyperbolic groups, small actions on $\R$-trees have bounded backtracking, and hence induce pseudo metrics in $\ov\calD_\G$ by Proposition \ref{prop.BBTtrees}.

\subsection{Geodesic currents and a conjecture of Bonahon}
Since the seminal work of Bonahon \cite{bonahon.currentsTeich}, there has been much interest in understanding, in various settings, which metrics admit continuous extensions to the space of currents. In the article \cite{bonahon.currentshypgroups}, Bonahon made a conjecture about which actions on $\R$-trees have a corresponding translation length function that extends continuously to the space of currents. More specifically, Bonahon conjectured that if the stable translation length function associated to an action on an $\R$-tree admits a continuous extension, then the action must be small. 

Our work allows us to produce a counter-example to this conjecture. To do so, we first prove the following result which we deduce from the fact that pseudo metrics in $\partial_M\scrD_\G$ can be represented, in some sense, as linear combinations of pseudo metrics in $\Dc_\G$, see Proposition \ref{prop.d+-inf}.

\begin{theorem}\label{prop.contextbound}
    Let $\G$ be hyperbolic and virtually torsion-free. Then for any $d\in \ov\calD_\G$, the stable translation length $\ell_d:\G\ra \R$ continuously extends to $\Curr(\G)$.
\end{theorem}
In a forthcoming preprint of Kapovich and Martinez-Granado, this result is obtained without the torsion-free assumption \cite{kap.margra}. Combining Theorem \ref{thm.examples} and Theorem \ref{prop.contextbound} we settle Bonahon's conjecture in the negative.
\begin{theorem}\label{Bonahoncounterex}
    There exist hyperbolic groups $\G$ for which there is a minimal, isometric action of $\G$ on an $\R$-tree $(T,d_T)$ such that
 \begin{enumerate}
        \item the action is not small; and,
        \item the stable translation length $\ell_T$ extends continuously to $\Curr(\G)$.
    \end{enumerate}
\end{theorem}
We will prove this result in Section \ref{sec.extensiontranslength}, where we produce examples of actions as in the theorem above from any hyperbolic group acting geometrically on a $\CAT(0)$ cube complex with a non-virtually cyclic hyperplane stabilizer. Theorem \ref{Bonahoncounterex} also suggests that for hyperbolic groups, the isometric actions on $\R$-trees to look at are not only the small ones, but also those having quasi-convex interval stabilizers. See Question \ref{qstn.Rtrees}. \\


The organization of the paper is as follows. Section \ref{sec.preliminaries} covers preliminary material concerning hyperbolic spaces and groups that we will need thorough the article. In Section \ref{sec.optqiconstant}, we prove Theorem \ref{thm.Dildistancelike} from which we deduce Theorem \ref{thm.dilatta}.
We then apply this result to obtain results regarding optimal growth rate constants for Anosov representations. Manhattan geodesics are constructed in Section \ref{sec.Manhattangeod} where we prove Theorem \ref{thm.Mangeo}. In this section, we obtain explicit formulas for the dilations of pairs of points in Manhattan geodesics and prove some properties about the geodesic bicombing of the Manhattan geodesics. In Section \ref{sec:manhattanboundary} we characterize the Manhattan boundary as the limits at infinity of Manhattan geodesics, in the form of Theorem   \ref{thm.Mboundaryvisible}. In addition, there we prove Proposition \ref{prop.chartransverse}, a criterion for a pair of points in the Manhattan boundary to be the points at infinity of a Manhattan geodesic. Examples of boundary metric structures are discussed in Section \ref{sec:examples}. 
In the final section we prove Theorem \ref{prop.contextbound} and discuss counterexamples to Bonahon's conjecture.\\

\noindent \textbf{Acknowledgements.}
The authors are grateful to Didac Martinez-Granado, Richard Canary and Karen Vogtmann for helpful comments. The second author also thanks Ian Agol for his support and encouragement throughout this project. Lastly we want to thank the anonymous referee for pointing out a mistake in a previous version of the paper and for providing useful comments and feedback that have improved the presentation of our work.
The second author was partially supported by the Simons Foundation (\#376200, Agol).


\section{Preliminaries}\label{sec.preliminaries}
\subsection{Gromov hyperbolic spaces and groups} \label{sec.hypspacesgroups}

Consider a pseudo metric space $(X,d)$. Recall that a pseudo metric is a function $d(\cdot,\cdot) : X^2 \to \mathbb{R}_{\ge 0}$ that satisfies the triangle inequality, is symmetric and has the property that $d(x,x)= 0 $ for all $x \in X$.
For each $z \in X$ the \emph{Gromov product} $( \cdot |\cdot)_z : X \times X \to \R_{\ge0}$ is defined as
\[
(x|y)_z = \frac{1}{2} (d(x,z) + d(z,y) - d(x,y)) \ \text{ for any $x,y \in X$.}
\]
When we deal with several pseudo metrics on $X$, we use the notation $(\cdot|\cdot)_{z,d}$ for the Gromov product with respect to $d$.
We say that $(X,d)$ is $\delta$-\emph{hyperbolic} if for every $x,y,z, w \in X$,
\[
(x|y)_z \ge \min\{(x|w)_z, (y|w)_z\} - \delta,
\]
and that $(X,d)$ is hyperbolic if it is $\delta$-hyperbolic for some $\delta \geq 0$.

Given pseudo metric spaces $(X,d_X)$ and $(Y,d_Y)$ we say that the function $F:X \ra Y$ is a \emph{quasi-isometric embedding} if there exist $\lam,C>0$ such that
\begin{equation*}
\frac{1}{\lam} d_X(x,y) - C \le d_Y(Fx,Fy) \le \lam d_X(x,y) + C \ \ \text{ for all $x,y \in X$.}
\end{equation*}
A quasi-isometric embedding $F:X \ra Y$ is a \emph{quasi-isometry} if in addition there is some $A\geq 0$ such that every point in $Y$ is within $A$ of some point in $F(X)$.
Two pseudo metrics $d, d_*$ on the same space $X$ are quasi-isometric if the identity map $(X,d)\ra (X,d_*)$ is a quasi-isometry. If there exist $\tau, C >0$ such that
\begin{equation}\label{eq.rs}
|\tau d(x,y) - d_*(x,y)|\le C \ \ \text{ for all 
$x,y \in X$}
\end{equation}
then we say that $d$ and  $d_*$ are \emph{roughly similar}, and \emph{roughly isometric} if \eqref{eq.rs} holds with $\tau=1$. 

A pseudo metric $d$ on $X$ is said to be \emph{geodesic} if every two elements in 
$X$ can be joined by an arc isometric to the interval of length equal to the distance between the two points. Given $\al\geq 0$, a pseudo metric $d$ is said to be $\alpha$-\emph{roughly geodesic} if for any $x, y \in X$ there is a sequence of points $x = x_0, \ldots, x_n = y \in X$  such that for all $0\le i \leq j \le n$
\begin{equation}\label{eq.rg}
    |j - i | -\alpha \le d(x_i,x_j) \le |j-i| + \alpha.
\end{equation}
Such a sequence $x_0, \ldots, x_n \in X$ for which \eqref{eq.rg} holds is referred to as an $\alpha$-\emph{rough geodesic}, or an $(\al,d)$-rough geodesic if we want to emphasize the dependence on $d$. A pseudo metric space is roughly geodesic if it is $\al$-rough geodesic for some $\al$.

\begin{remark}
By abusing notation, we also extend the definition of rough similarity/isometry to functions on $X\times X$ that are not necessarily pseudo metrics. Similarly, we can talk about nonnegative functions on $X \times X$ being hyperbolic or roughly geodesic.
\end{remark}
Hyperbolicity can also be characterised using quasi-centers. Given $\al,\del \geq 0$, there exists $\k=\k(\al,\del)$ such that any triple of points $x,y,z$ in the $\del$-hyperbolic, $\al$-rough geodesic pseudo metric space $(X,d)$ has a $\k$-\emph{quasi-center}. That is, there is a point $p \in X$ such that
\begin{equation*}
\max \left\{(x|y)_{p, d},(y | z)_{p, d},(z | x)_{p, d}\right\} \leq \k.
\end{equation*}
We say that $p$ is a $\left(\k , d \right)$-quasi-center if we want to make explicit that $p$ is a $\k$-quasi-center with respect to the pseudo metric $d$.

We will also require the following result \cite[Prop.~15~(i),~Ch.~5]{ghys.delaharpe}.

\begin{proposition}\label{prop.qiGromovprod}
For all $\al,\del,\ep\geq 0$ and $\lam_1,\lam_2>0$ there exists some $C\geq 0$ such that the following holds. Let $(X,d_X), (Y,d_Y)$ be $\del$-hyperbolic and $\al$-rough geodesic pseudo metric spaces, and let $F:X\ra Y$ satisfy
\begin{equation*}
    \frac{1}{\lam_1}d_X(x,y)-\ep\leq d_Y(Fx,Fy)\leq \lam_2d_X(x,y)+\ep
\end{equation*}
for all $x,y\in X$. Then for all $x,y,w\in X$:
\begin{equation*}
    \frac{1}{\lam_1}(x|y)_{w,d_X}-C\leq (Fx|Fy)_{Fw,d_Y}\leq \lam_2(x|y)_{w,d_X}+C.
\end{equation*}
\end{proposition}

\subsection{Hyperbolic groups} 
 Suppose that $\G$ is a finitely generated group. Let $S \subset \G$ be a, not necessarily symmetric, set of elements that generates $\G$ as a semi-group. From this set we can equip $\G$ with the corresponding word length function $|\cdot|_S : \G \to \R_{\ge 0}$ that assigns to a group element $x$ the length of the shortest word(s) that represents $x$ with letters in $S$, i.e.
\[
|x|_S = \min\{ n \in \Z_{\ge 0} : x = s_1 \cdots s_n \ \text{ with $s_1,\ldots,s_n \in S$} \} \ \text{ for each $x \in \G$}.
\]
By convention, the identity is assigned word length $0$. The word metric is the (not necessarily symmetric) distance $d_S(\cdot,\cdot):\G \times \G \ra \R_{\geq 0}$ given by 
$$d_S(x,y):=|x^{-1}y|_S \  \text{ for }x,y\in \G.$$
We say that $\G$ is \emph{hyperbolic} if for some finite, symmetric generating set $S$, $(\G, d_S)$ is a hyperbolic metric space. All hyperbolic groups we consider will implicitly be assumed to be \emph{non-elementary}, i.e. we will assume  that they do not contain a finite index cyclic subgroup. In general, we will say that a pseudo metric $d$ on $\G$ is hyperbolic (resp.~roughly geodesic) if $(X,d)$ is a hyperbolic (resp.~roughly geodesic) pseudo metric space.

As discussed in the introduction we will be interested in the collection $\Dc_\G$ of hyperbolic pseudo metrics on $\G$ that are quasi-isometric to a word metric and that are $\G$-invariant: $d( h x, hy ) = d(x,y)$ for all $h,x,y \in \G$.
Pseudo metrics in $\Dc_\G$ are necessarily roughly geodesic \cite[Thm.~1.10]{bhm}. We will use the notation $h(d)$ to denote the exponential growth rate of $d\in \calD_\G$,
\[
h(d) = \limsup_{n\to\infty} \frac{1}{n} \log \# \{x\in\G: d(o,x) < n \}=\limsup_{n\to\infty} \frac{1}{n} \log \# \{[x]\in\conj: \ell_d[x] < n \},
\]
which is always finite and strictly positive.
\begin{example}
    Pseudo metrics belonging to $\Dc_\G$ include: word metrics for finite, symmetric generating sets, orbit pseudo metrics associated to cocompact, isometric and properly discontinuous actions on hyperbolic geodesic metric spaces, and Green metrics associated to finitely supported symmetric random walks that visit the whole group \cite[Cor.~1.2]{bhm}.
\end{example}

\subsection{The Manhattan curve} \label{sec.manhattan}
Consider two pseudo metrics $d,d_\ast \in \Dc_\G$. 
We define the \emph{Manhattan curve} associated to this pair to be the boundary of the convex set
\[
\mathcal{C}_{d_\ast/d}^M = \left\{ (a,b) \in \R^2 : \sum_{x\in\G} e^{-a d_\ast(o,x) - b d(o,x)} < \infty \right\}.
\]
Convexity of $\mathcal{C}_{d_\ast/d}^M$ follows from H\"older's inequality. Manhattan curves were first introduced by Burger for the displacement functions associated to actions on rank 1 symmetric spaces \cite{burger}. In the current setting, $\mathcal{C}_{d_\ast/d}^M$ was studied by Cantrell and Tanaka in \cite{cantrell.tanaka.2} and \cite{cantrell.tanaka.Man} when $d,d_\ast$ are metrics in $\Dc_\G$. In these works regularity and rigidity results pertaining to these curves were obtained. For example, Theorem 1.1 in \cite{cantrell.tanaka.Man} states that $\calC^M_{d_\ast/d}$ is a straight line if and only if $d$ and $d_\ast$ are roughly similar.

The Manhattan curve for $d,d_\ast$ can be parameterised using a function $\theta_{d_\ast/d} :\R\to\R$ which we define in the following way. For each $t \in \R$, let $\theta_{d_\ast/d}(t)$ be the abscissa of convergence of 
\[
\sum_{x\in\G} e^{-td_\ast(o,x) - sd(o,x)}
\]
as $s$ varies and $t$ remains fixed. That is, fixing $t$ and allowing $s$ to vary, the above series converges for $s > \theta_{d_\ast/d}(t)$ and diverges for $s < \theta_{d_\ast/d}(t)$.

\begin{remark}\label{rmk.conj=distMC}
An equivalent way of obtaining the Manhattan curve is to use the stable translation length functions $\ell_d$, $\ell_{d_\ast}$ and to count over conjugacy classes. In particular, in either of definitions of the Manhattan curve above, either via the set $\mathcal{C}^M_{d_\ast/d}$ or via the parameterisation $\theta_{d_\ast/d}$, if one replaces the metrics and counting over group elements with the stable translation length functions and counting over conjugacy classes, we obtain the same curve \cite[Prop.~3.1]{cantrell.tanaka.Man}.
\end{remark}

Convex functions from $\R$ to $\R$ are continuous and differentiable Lebesgue almost everywhere. Cantrell and Tanaka showed Manhattan curves have better regularity.
\begin{theorem}[{\cite[Thm.~1.1]{cantrell.tanaka.Man}}] \label{thm.manreg}
    Let $\theta_{d_\ast/d}$ be the Manhattan curve for $d,d_\ast \in \Dc_\G$ as defined above. Then $\theta_{d_\ast/d}$ is strictly decreasing, convex, and $C^1$, i.e. it has continuous first derivative.
\end{theorem}
\begin{remark}
    In fact, for certain pairs of pseudo metrics $d,d_\ast \in \Dc_\G$, the associated Manhattan curve is known to be analytic \cite{cantrell.tanaka.2}. This is the case for pairs of word metrics for example.
\end{remark}
Note that in \cite{cantrell.tanaka.Man} this result was proved for metrics opposed to pseudo metrics, however the same proof applies to the above case. Theorem \ref{thm.manreg} is critical to our arguments and we will use it implicitly throughout this work.

\subsection{Geodesic currents}\label{subsec:currents}
Let $\partial \G$ denote the Gromov boundary of $\G$, which is an infinite, compact metrizable space since $\G$ is non-elementary. The double boundary is the set $\partial^2 \G$ of unordered pairs of distinct points in $\G$, endowed with the expected topology and the diagonal action of $\G$. A \emph{geodesic current} on $\G$ is a locally finite, $\G$-invariant measure $\mu$ on $\partial^2 \G$, meaning that $\mu(K)$ is finite for any compact subset $K\subset \partial^2 \G$. We let $\Curr(\G)$ denote the space of geodesic currents equipped with the weak$^*$ topology. Geodesic currents were introduced by Bonahon, first for surface groups \cite{bonahon.currentsTeich} and later for general hyperbolic groups \cite{bonahon.currentshypgroups}. 

The space of geodesic currents can be thought of as a completion of the space of conjugacy classes of $\G$ in the following sense: if $[x]$ is the conjugacy class of the non-torsion element $x\in \G$, then $x=y^n$ for $y\in \G$ a primitive element and $n\neq 0$. If $y_{\infty}, y_{-\infty}$ denote the two points in $\partial \G$ that are fixed by $y$, then the set $\calA_{[y]}=\cor{\cor{gy_{-\infty},gy_{+\infty}} \colon g\in \G}$ is a discrete, $\G$-invariant subset of $\partial^2\G$. In this way, the \emph{rational} current associated to $[y]$ is given by
$$\eta_{[y]}=\sum_{\cor{p,q}\in \calA_{[y]}}{\delta_{\cor{p,q}}},$$
and similarly we define $\eta_{[x]}=|n|\eta_{[y]}$. The set $\cor{\lam \eta_{[x]}\colon \lam>0,[x]\in \conj'}$ turns out to be dense in $\Curr(\G)$ \cite[Thm.~7]{bonahon.currentshypgroups}.

By considering $\bbP\Curr(\G):=(\Curr(\G)\bs \cor{0})/\R^+$, where the action of $\R^+$ is given by scalar multiplication, we obtain the space of \emph{projective geodesic currents}, which is compact and metrizable when equipped with the quotient topology \cite[Prop.~6]{bonahon.currentshypgroups}.

\section{Optimal quasi-isometry constants}\label{sec.optqiconstant}
In this section, we prove Theorem \ref{thm.dilatta}. 
It will follow from Theorem \ref{thm.Dildistancelike}, which works in the more general setting of hyperbolic distance-like functions defined below. We also apply this result to norm and eigenvalue functions associated to Anosov representations in Section \ref{sec.Anosov}.

\subsection{Hyperbolic distance-like functions}
\begin{definition}
    By a \emph{hyperbolic distance-like function} on $\G$, we mean a function $\psi:\G \times \G \ra \R$ satisfying the following:
    \begin{enumerate}
        \item positivity: $\psi(x,y)\geq 0$ and $\psi(x,x)=0$ for all $x,y\in \G$;
        \item the triangle inequality: $\psi(x,z)\leq \psi(x,y)+\psi(y,z)$ for all $x,y,z\in \G$;
        \item $\G$-invariance: $\psi(sx,sy)=\psi(x,y)$ for all $x,y,s\in \G$; and,
        \item for any $d_0\in \calD_\G$ and $C\geq 0$ there exists $D\geq 0$ such that the following holds: if $x,y, w\in \G$ are such that $(x|y)_{w,d_0}\leq C$, then their Gromov product for $\psi$ satisfies $$(x|y)_{w,\psi}:=\frac{(\psi(x,w)+\psi(w,y)-\psi(x,y))}{2} \leq D.$$
    \end{enumerate}
\end{definition}
Note that pseudo metrics belonging to $\ov\calD_\G$ are hyperbolic distance-like functions. This is also the case for the logarithm of the norms of Anosov representations of $\G$, see Lemma \ref{lem.An-hdlf}.

For our purposes, a key property is that hyperbolic distance-like functions are roughly geodesic, in the sense that they satisfy condition \eqref{eq.rg}.

\begin{lemma}\label{lem.RGforHDLF}
    Let $\G$ be a non-elementary hyperbolic group. If $\psi:\G \times \G \ra \R$ is a hyperbolic distance-like function, then $\psi$ is roughly geodesic. In particular, every pseudo metric belonging to $\ov\calD_\G$ is roughly geodesic. Moreover, for every $d_0\in \calD_\G$ and $\al_0\geq 0$ there is some $\al\geq 0$ such that if  $x=z_0,\dots,z_m=y$ is an $(\al_0,d_0)$-rough geodesic, then we can find a non-decreasing subsequence $0=i(0)\leq i(1)\leq \cdots \leq  i(n)=m$ such that $z_{i(0)},\dots,z_{i(n)}$ is an $(\al,\psi)$-rough geodesic.  
\end{lemma}

\begin{remark}\label{rmk.hdlfhyp}
By Lemma \ref{lem.RGforHDLF} and the invariance of hyperbolicity under quasi-isometry among roughly geodesic spaces, we deduce that all symmetric hyperbolic distance-like functions that are quasi-isometric to a word metric are hyperbolic pseudo metrics, and hence lie in $\calD_\G$.  In Corollary \ref{cor.hyp} we will see that in fact \emph{every} pseudo metric in $\ov\calD_\G$ is hyperbolic. 
\end{remark}

Lemma \ref{lem.RGforHDLF} requires the following simple lemma, the proof of which we leave to the reader. 
\begin{lemma}\label{lemma:seqRG} Let $a_0,\dots,a_m$ be a sequence of real numbers and $L \geq 0$ be such that $a_0\leq a_m$ and $|a_i-a_{i+1}| \leq L$ for $0 \leq i<m$. Assume that $[a_0, a_m] \cap \mathbb{Z}=\{k, k+1, \ldots, k+n\}$. Then there exists a non-decreasing subsequence $0=i(0) \leq i(1) \leq \cdots \leq i(n)=m$ such that $\left|(k+j)-a_{i(j)}\right| \leq (L+2)/ 2$ for all $0 \leq j \leq n$.
\end{lemma}

\begin{proof}[Proof of Lemma \ref{lem.RGforHDLF}] Let $d_0\in \calD_\G$, and let $x=z_0, \dots, z_m=y$ be an $(\al_0,d_0)$-rough geodesic. Let $L:=\max \{\psi(o, u) \colon d_0(o, u) \leq 1+\alpha_0\}$, which is finite because $\G$ is finitely generated and $d_0$ is proper. Since $\psi$ is $\Gamma$-invariant and satisfies the triangle inequality, the sequence $a_i :=\psi(z_0, z_i)$ satisfies the assumptions of Lemma \ref{lemma:seqRG}. Therefore, for each integer $j$ between $0$ and $\psi(o, y)$ there exists some $x_j:=z_{i(j)}$ with $\left|\psi(x, x_j)-j\right| \leq(L+2) / 2$, and such that $i(j) \leq i(j+1)$ for all $j$. Let us suppose this sequence is $x=x_0, \dots, x_n=y$, which lies in an $(\al_0,d_0)$-rough geodesic. We obtain $\left(x| x_j\right)_{x_i,d_0} \leq 3 \al_0 / 2$ for $0 \leq i \leq j \leq n$. But $\psi$ is a hyperbolic distance-like function, implying that $\left(x|x_j\right)_{x_i,\psi} \leq D$ for some $D$ independent of the sequence. This translates to
\[ \psi(x_i, x_j) \leq 2 D+\psi(x, x_j)-\psi(x, x_i) \leq j-i+2 D+L+2\]
for all $0\leq i\leq j\leq n$. On the other hand, by the definition of the $x_i$'s and the triangle inequality we have
\[\psi(x_i, x_j) \geq \psi(x, x_j)-\psi(x, x_i) \geq j-i-(L+2)\]
for all $0\leq i\leq j\leq n$, so that $x_0,\dots,x_n$ is an $(\al,\psi)$-rough geodesic with $\alpha:=2 D+L+2$.
\end{proof}
 
Another important property of hyperbolic distance-like functions is that they can be coarsely approximated by word metrics. The proof of the next result follows exactly as in the proof of \cite[Lem.~5.1]{oregon-reyes.ms} which shows that pseudo metrics in $\Dc_\G$ can be approximated by word metrics. We leave the details to the reader. 
\begin{lemma}\label{lem:HDLFapproxword} Let $\psi$ be a hyperbolic distance-like function that is $\alpha$-roughly geodesic. For $n>\alpha+1$, let $S_n:=\{x \in \Gamma \mid \psi(o, x) \leq n\}$. Then $S_n$ generates $\Gamma$ as a semigroup and for all $x \in \Gamma$ we have
$$
(n-\alpha-1)|x|_{S_n}-(n-1) \leqslant \psi(o, x) \leq n|x|_{S_n} .
$$
\end{lemma}

In the lemma above the sets $S_n$ are not necessarily finite since $\psi$ might not be proper, i.e. the balls $\{x\in \G: \psi(o,x)\leq R\}$ are not required to be finite. However, if $\psi$ is symmetric and proper, then it is quasi-isometric to a word metric for a finite generating set, and hence it belongs to $\calD_\G$. This implies the following.

\begin{corollary}\label{coro.calDiffproper}
    Let $d$ be a pseudo metric in $\ov\calD_\G$. Then $d\in \calD_\G$ if and only if $d$ is proper, meaning that the balls $\{x\in \G: d(x,o)\leq R\}$ are finite for all $R>0$.
\end{corollary}

A criterion for properness is given by the following lemma, which will be used in the proof of Proposition \ref{prop.da} and might be of independent interest.

\begin{lemma}\label{lem.proper}
    Let $\G$ be a non-elementary hyperbolic group and let $(d_n)_n$ be a sequence of (not necessarily hyperbolic) left-invariant pseudo metrics pointwise converging to the pseudo metric $d_\infty$. If $\limsup_{T \to \infty}{\frac{\log\# \{x\in \G: d_n(o,x)<T\}}{T}}=1$ for each $n$, then $d_\infty$ is proper.
\end{lemma}

\begin{proof}
  We begin with the following observation: if $S \subset \G$ is a finite and symmetric generating set, then its exponential growth rate $h_S=h(d_S)$ satisfies $h_S\leq \max_{x\in S}{d_\infty(x,o)}$. To show this, for $\ep>0$ we consider $n$ large enough so that $\max_{x\in S}{d_n(x,o)}\leq \max_{x\in S}{d_\infty(x,o)}+\ep$. Then for any $x\in \G$ we have
   \[d_n(x,o)\leq (\max_{x\in S}{d_\infty(x,o)}+\ep)d_S(x,o),\]
 and hence $h_S \leq \max_{x\in S}{d_\infty(x,o)}+\ep$ since the exponential growth rate of $d_n$ is 1. The observation then follows by letting $\ep$ tend to zero.

 Now we start the proof of the lemma and suppose for the sake of a contradiction that $S_\infty=\{x\in \G : d_\infty(x,o)\leq R\}$ is infinite for some $R>0$. Up to increasing $R$ we can assume that $S_\infty$ generates $\G$. We take an infinite nested sequence $S_1\subset S_2\subset \dots$ of finite and symmetric generating subsets of $\G$ contained in $S_\infty$ and such that $\# S_i$ tends to infinity. By our previous observation we have $h_{S_i}\leq R$ for all $i$. On the other hand, by \cite[Thm.~1]{AL} there exists $\al>0$ such that $h_{S_i}\geq \log (\al (\# S_i))$ for all $i$, which is our desired contradiction.
\end{proof}

\subsection{Proof of the main result} Similarly as in the case of pseudo metrics in $\calD_\G$, for hyperbolic distance-like functions $\psi,\psi_*$ on $\G$ we can define the stable translation length function
\begin{equation*}
    \ell_\psi[x]:=\lim_{n\to \infty}{\frac{1}{n}\psi(o,x^n)} \hspace{2mm} \text{for } x\in \G,
\end{equation*}
and the dilation of $\psi$ and $\psi_*$:
\begin{equation*}
    \Dil(\psi,\psi_*):=\inf\{\lam>0 \colon \ell_\psi[x]\leq \lam \ell_{\psi_*}[x] \text{ for all } [x]\in \conj \},
\end{equation*}
where we define $\Dil(\psi,\psi_*)=\infty$ if no such $\lam$ exists. Our main result states that for hyperbolic distance-like functions, the dilation is the optimal quasi-isometry constant.
\begin{theorem}\label{thm.Dildistancelike}
    Let $\G$ be a non-elementary hyperbolic group and let $\psi,\psi_*$ be hyperbolic distance-like functions on $\G$ such that $\Dil(\psi_*,\psi)<\infty$. Then there exists some $C\geq 0$ such that for any $x,y\in \G$
    \begin{equation*}
        \psi_*(x,y)\leq \Dil(\psi_*,\psi)\psi(x,y)+C.
    \end{equation*}
\end{theorem}
From this result, we immediately deduce Theorem \ref{thm.dilatta} from the introduction.
\begin{proof}[Proof of Theorem \ref{thm.dilatta}]
    Let $d,d_*\in \calD_\G$, so that they are hyperbolic distance-like functions satisfying $\Dil(d,d_*), \Dil(d_*,d)<\infty$. Then by Theorem \ref{thm.Dildistancelike} there is some $A\geq 0$ such that 
    \begin{equation*}
        \Dil(d,d_*)^{-1}d(x,y)-A\leq d_*(x,y)\leq \Dil(d_*,d)d(x,y)+A
    \end{equation*}
    for all $x,y\in \G$. The conclusion then follows from Proposition \ref{prop.qiGromovprod}.
\end{proof}

For the proof of Theorem \ref{thm.Dildistancelike} we fix a finite, symmetric generating set $S\subset \G$ with word metric $d_S$, and let $\ell_S$ denote the stable translation length for this metric. Similarly, $(\cdot|\cdot)_S$ denotes the Gromov product for $d_S$ based at the identity, and $|\cdot|_S$ denotes the word length with respect to $S$. Let $\del$ be a hyperbolicity constant for $d_S$, which we assume is positive.
We start with some lemmas.

\begin{lemma}\label{lem.isomloxo} Let $g\in \G$ and $D>0$ satisfy 
\begin{equation}\label{eq.goodposition}
    (g^{-1}|g)_S<D \ \text{ and } \ |g|_S>2D+4\del.
\end{equation} Then
\[(g^{-m}|g^n)_S<D+2\del \hspace{2mm} \text{ for all } m,n\geq 1. \]
\end{lemma}

\begin{proof}
    Let $a:=|g^2|_S-|g|_S-2\del$. The inequality $(g^{-1}|g)_S<D$ is equivalent to $2|g|_S-|g^2|_S<2D$, so that $a>|g|_S-2\del-2D>2\del>0$ by \eqref{eq.goodposition}. From the proof of \cite[Thm.~1.1]{oregon-reyes.ineq} and the fact that $\del$ is positive we deduce that $a+|g^n|_S\leq |g^{n+1}|_S$ for all $n\geq 1$. In particular, we have 
    \[(g^{-m}|g^{-1})_S=(|g^m|_S+|g|_S-|g^{m-1}|_S)/2\geq (|g|_S+a)/2>|g|_S/2>D+2\del\] for all $m>1$, and similarly $(g |g^{n})_S>D+2\del$
    for all $n>1$. But by $\del$-hyperbolicity and \eqref{eq.goodposition} we have
    \[D+2\del>(g^{-1}|g)_S+2\del \geq \min\{(g^{-1}|g^{-m})_S,(g^{-m}|g^n)_S,(g^n|g)_S\}\]
    which implies $(g^{-m}|g^n)_S<D+2\del$ for all $m,n\geq 1$, as desired.
\end{proof}

\begin{lemma}\label{lem.closetogeodword}
There exist constants $C',R'\geq 0$ such that for any $x\in \G$ there is some $\gam_x\in \G$ such that $d_S(x,\gam_x)\leq R'$ and
        $$(\gam_x^{-m}|\gam_x^{n})_{S}\leq C' \hspace{2mm} \text{ for all } m,n\geq 0.$$
\end{lemma}

\begin{proof}
Let $u,v\in \G$ be a \emph{ping pong pair} in the sense of \cite[Sec.~2.2.4]{DGLM}. The proof of \cite[Lem.~2.2.6]{DGLM} actually shows the following: if $x\in \G$ satisfies $(x^{-1}|x)_S\geq \max\{|u|_S,|v|_S\}+30\del$, then there exists $s\in \{u,v\}$ such that 
\begin{equation}\label{eq.xs}
    ((xs)^{-1}|xs)_S\leq \frac{1}{2}\max\{|u|_S,|v|_S\} -7\del.
\end{equation}
Now we consider three cases. If $(x^{-1}|x)_S<\max\{|u|_S,|v|_S\}+30\del$ and $|x|_S>2\max\{|u|_S,|v|_S\}+64\del$, then Lemma \ref{lem.isomloxo} applied to $g=x$, and $D=\max\{|u|_S,|v|_S\}+30\del$ gives us 
\[(x^{-m}|x^n)_S<\max\{|u|_S,|v|_S\}+32\del\]
for all $m,n\geq 0$, and we choose $\gam_x=x$. 

Next, if $(x^{-1}|x)_S\geq \max\{|u|_S,|v|_S\}+30\del$ and $|x|_S>2\max\{|u|_S,|v|_S\}+64\del$, let $s\in \{u,v\}$ satisfy \eqref{eq.xs}. 
Then $|xs|_S\geq |x|_S-\max\{|u|_S,|v|_S\}>\max\{|u|_S,|v|_S\}-10\del$, and Lemma \ref{lem.isomloxo} applied to $g=xs$ and $D=\max\{|u|_S,|v|_S\}/2-7\del$ gives us
\[((xs)^{-m}|(xs)^n)_S<\frac{1}{2}\max\{|u|_S,|v|_S\}-5\del\]
for all $m,n\geq 0$. In this case we choose $\gam_x=xs$. 

Finally, if $|x|_S\leq 2\max\{|u|_S,|v|_S\}+64\del$ then we choose $\gam_x=o$. In conclusion, the lemma follows with $R'=2C'=2\max\{|u|_S,|v|_S\}+64\del$.
\end{proof}

\begin{lemma}\label{lem.gam_x}
    Let $\gam:\G\ra \G$, $x \mapsto \gam_x$ be the assignment from Lemma \ref{lem.closetogeodword}. Then for any hyperbolic distance-like function $\psi$ on $\G$ there exist $C_0, R_0\geq 0$ such that for any $x\in \G$ we have $\max(\psi(x,\gam_x), \psi(\gam_x,x))\leq R_0$ and 
    $$(\gam_x^{-m}|\gam_x^{n})_{o,\psi}\leq C_0$$
    for all $m,n\geq 0$. In particular we have $\psi(o,\gam_x)\leq \ell_\psi[\gam_x]+2C_0$ for all $x\in \G$.
\end{lemma}
As before, we are using $o\in\G$ to denote the identity element.
\begin{proof}
    Let $d_S\in \calD_\G$ be a word metric with constants $C',R'$ given by Lemma \ref{lem.closetogeodword}.  Then for any $x\in \G$ we have $d_S(x,\gam_x)\leq R'$ and $(\gam_x^{-m}|\gam_x^{n})_{S}\leq C'$ for all $m,n\geq 0$.
    If $\psi$ is any hyperbolic distance-like function, let $C_0\geq 0$ (resp. $R_0\geq 0$) be such that $(p|q)_{S}\leq C'$ (resp. $(p|q)_{S}\leq R'$) implies $(p|q)_{o,\psi}\leq C_0$ (resp. $(p|q)_{o,\psi}\leq R_0/2$) for all $p,q\in \G$. Therefore, for all $x\in \G$ we have 
    $\max(\psi(x,\gam_x),\psi(\gam_x,x))\leq 2(x^{-1}\gam_x|x^{-1}\gam_x)_{o,\psi}\leq R_0$
    and \begin{equation}\label{eq.psiadditive}\psi(o,\gam_x^{m})+\psi(o,\gam_x^{n})\leq \psi(o,\gam_x^{m+n})+2C_0\end{equation}
    for all $m,n\geq 0$, which proves the first assertion of the proposition. For the second assertion, we apply \eqref{eq.psiadditive} to $m=1$ and obtain
    $\psi(o,\gam_x)\leq \psi(o,\gam_x^{n+1})-\psi(o,\gam_x^{n})+2C_0$ for all $n$. By adding these inequalities for $0\leq n\leq k$ we get
    $$(k+1)\psi(o,\gam_x)\leq \psi(o,\gam_x^{k+1})+2(k+1)C_0.$$
    The proof concludes after dividing by $(k+1)$ and letting $k$ tend to infinity. 
\end{proof}

\begin{corollary}\label{coro.ssdist<mls}
    There exists a finite set $B\subset \G$ such that the following holds. Given $\psi$ a hyperbolic distance-like function on $\G$, there exists $C_1\geq 0$ such that for any $x,y\in \G$ we have
    \begin{equation}\label{eq.dist<mls}
        \psi(x,y)\leq \max_{u\in B}{\ell_\psi[x^{-1}yu]}+C_1.
    \end{equation}
\end{corollary}
\begin{remark}
    The inequality \eqref{eq.dist<mls} considerably refines \cite[Prop.~3.1]{oregon-reyes.ms}, where the maximum on the left hand side was multiplied by $(1-\ep)$ for arbitrary $\ep \in (0,1)$, but the additive error on the right hand side depended on $\ep$.
\end{remark}
\begin{proof}
Let $R'\geq 0$ and $x\mapsto\gam_x$ be as in Lemma \ref{lem.closetogeodword}. Suppose they are induced by the generating set $S\subset \G$, and set $B:=\{u\in \G \colon d_S(o,u)\leq R'\}$. If $\psi$ is a hyperbolic distance-like function on $\G$ and $C_0,R_0$ are the constants found in Lemma \ref{lem.gam_x}, we set $C_1:=2C_0+R_0$. Since for all $x,y\in \G$ we have $y^{-1}x\gam_{x^{-1}y}\in B$, we deduce
$$\psi(x,y)=\psi(o,x^{-1}y)\leq \psi(o,\gam_{x^{-1}y})+R_0 \leq \ell_\psi[\gam_{x^{-1}y}]+2C_0+R_0\leq \max_{u\in B}{\ell_\psi[x^{-1}yu]}+C_1. \qedhere$$
\end{proof}

\begin{proof}[Proof of Theorem \ref{thm.Dildistancelike}]
    Let $B\subset \G$ be the finite set given by Corollary \ref{coro.ssdist<mls}, and let $C_1$ be the corresponding constant for $\psi_*$. Then for any $x,y \in \G$ we get
    \begin{equation*}
    \begin{aligned}
          \psi_*(x,y)& \leq \max_{u\in B}{\ell_{\psi_*}[x^{-1}yu]}+C_1 \\
         & \leq  \Dil(\psi_*,\psi)\cdot \max_{u\in B}{\ell_{\psi}[x^{-1}yu]}+C_1 \\
        & \leq \Dil(\psi_*,\psi)\cdot \psi(x,y)+\left(C_1+\Dil(\psi_*,\psi)\cdot \max_{u\in B}{\psi(o,u)}\right), 
    \end{aligned}
    \end{equation*}
    This concludes the proof since the last term on the right-hand side is finite and  independent of $x,y$.
\end{proof}

\subsection{Anosov representations}\label{sec.Anosov}
At the beginning of the section, we introduced hyperbolic distance-like functions. Our reason for working at this level of generality was so that we could apply our methods to distance-like functions that are not necessarily pseudo metrics in $\Dc_\G$. In this section, we briefly present some applications to Anosov representations.

Let $\G$ be a finitely generated group equipped with a generating set $S$. A representation $\rho: \G \to \SL_m(\R)$ is said to be $j$\emph{-dominated} for $j \in \{1,\ldots, m-1\}$ if there exist constants $C, \mu >0$ such that
\[
\frac{\sigma_j(\rho(x))}{\sigma_{j+1}(\rho(x))} \ge Ce^{\mu |x|_S} \ \ \text{ for all $x\in\G.$}
\]
Here, for $A \in \SL_m(\R)$, $\sigma_1(A) \ge \sigma_2(A)\ge \ldots \ge \sigma_m(A)$ represent the singular values of $A$. This condition was studied by Bochi, Potrie and Sambarino in \cite{bps} in which it was shown that being $1$-dominated is equivalent to being \emph{projective Anosov} as defined by Labourie in \cite{labourie} and extended to all groups in \cite{gw}.
Further it is known that for a group to admit a $1$-dominated representation, it must be hyperbolic \cite[Thm.~3.2]{bps}. As we continue we will stop using the term $1$-dominated representations and will instead use projective Anosov.

The next lemma follows immediately from \cite[Lem.~7.1]{cantrell.tanaka.2}.

\begin{lemma}\label{lem.An-hdlf}
If $\rho:\G \ra \SL_m(\R)$ is a projective Anosov representation, then the function $$\psi_\rho(x,y):=\log \|\rho(x^{-1}y)\|$$ defines a hyperbolic distance-like function on $\G$, quasi-isometric to any pseudometric belonging to $\calD_\G$.
\end{lemma}

This lemma combined with Theorem \ref{thm.Dildistancelike} implies the following results for projective Anosov representations.
\begin{proposition} \label{prop.AR}
    Suppose that $\rho$ and $\rho_*$ are two projective Anosov representations (not necessarily of the same dimension). Then there exists $C >0$ such that for every $x\in\G$
    \[
    \Dil(\psi_\rho,\psi_{\rho_*})^{-1} \log\|\rho(x)\| - C \le \log\|\rho_*(x)\| \le \Dil(\psi_{\rho_*},\psi_\rho) \log\|\rho(x)\| + C.
    \]
\end{proposition}

\begin{proposition}\label{prop.arwm}
Suppose that $\rho: \G \to \SL_m(\R)$ is projective Anosov and $S\subset \G$ is any (not necessarily symmetric) finite generating set. Then there exists a constant $C >0$ such that
\[
\Dil(d_S,\psi_\rho)^{-1}|x|_S - C \le \log\|\rho(x)\|\le \Dil(\psi_\rho,d_S)|x|_S + C
\]
for all $x\in\G$.
\end{proposition}

\begin{remark}
    It may be possible to prove these results using ideas involving the semi-simplification representation used by Tsouvalas in \cite{tsouvalas}.
\end{remark}
These results show that there is a strong relationship between the eigenvalue and norm maps associated to Anosov representations. This complements the spectral rigidity results of Bridgeman, Canary, Labourie and Sambarino \cite{BCLS} and Cantrell and Tanaka \cite{cantrell.tanaka.2}. In particular it was shown, although not explicitly stated in \cite{cantrell.tanaka.2} that if $\rho, \rho_*$ are projective Anosov representations as above then the following are equivalent:
\begin{enumerate}
    \item there is $\tau>0$ such that $\log\lambda_1(\rho(x)) = \tau\log\lambda_1(\rho_*(x))$ for all $x \in \G$; and,
    \item there exist $\tau, C' >0$ such that $|\log\|\rho(x)\| - \tau\log\|\rho_*(x)\|| \le C'$ for all $x\in\G$.
\end{enumerate}
Further equivalence statements can be added to this list, see \cite{BCLS}. Proposition \ref{prop.AR} above implies this result and further shows that the dilations provide uniformly good upper and lower multiplicative bounds for the comparison between $\log\|\rho(x)\|$ and $\log\|\rho_*(x)\|$.

Our work also has applications to Borel Anosov representations: representations that are fully dominated, i.e. $\rho: \G \to \SL_m(\R)$ is $j$-dominated for every $j=1,\ldots m-1$. Representations coming from higher Teichm\"uller theory (i.e. Hitchin representations) are Borel Anosov. Given a representation $\rho : \G \to \SL_m(\R)$ we will use
$\lambda: \SL_m(\R) \to \R^m$ and $\mu: \SL_m(\R) \to \R^m$ for the Jordan and singular value projections. That is, $\lambda(\rho(x))$ and $\mu(\rho(x))$ are given by
\[
 (\log\lambda_1(\rho(x)), \ldots, \log\lambda_m(\rho(x))) \text{ and } (\log\sigma_1(\rho(x)), \ldots, \log\sigma_m(\rho(x)))
\]
respectively where $\lambda_j : \SL_m(\R) \to \R$ for $j=1,\ldots, m$ map a matrix to the absolute value of its $j$th largest (by modulus) eigenvalue.
A representation is Borel Anosov if and only if each of its exterior power representations is projective Anosov. Therefore, by Lemma \ref{lem.An-hdlf} we deduce that for a Borel Anosov representation $\rho:\G \ra \SL_m(\R)$ each of the functions $\psi_j:\G \times \G \ra \R$
$$\psi_j(x,y)=\log \sigma_1(\rho(x^{-1}y))+\dots+\log\sigma_j(\rho(x^{-1}y))$$
is a hyperbolic distance-like function on $\G$ for $1\leq j \leq m-1$.
This fact combined with Lemma \ref{lem.gam_x} gives us the following.

\begin{proposition} \label{prop.AR2}
Suppose that $\rho : \G \to \SL_m(\R)$ is Borel Anosov. 
Fix a norm $\|\cdot\|$ on $\R^m$ and a generating set $S$ for $\G$.
Then there exists a constant $C_0 >0$ such that for any $x \in \G$ there exists $y\in\G$ such that all of
\[
|x^{-1}y|_S, \ \ | \ell_S[y]- |y|_S |, \ \ \|\sigma(\rho(x))-\sigma(\rho(y))\|, \ \text{ and } \| \lambda(\rho(y)) - \mu(\rho(y))\|
\]
are at most $C_0$.
\end{proposition}

We can use this result to compare the optimal decay constants for the quotients of singular values and eigenvalues for Borel Anosov representations. Fix a group $\G$ and generating set $S$.
Given $C, \mu > 0$ we say that a representation $\rho: \G \to \SL_m(\R)$ is $(C, \mu)$ $j$-\emph{dominated}, respectively $(C, \mu)$ $j$-\emph{eigenvalue dominated} if for all $x\in\G$
\[
\frac{\sigma_{j+1}(\rho(x))}{\sigma_{j}(\rho(x))} \le Ce^{-\mu |x|_S}, \ \text{respectively} \ \frac{\lambda_{j+1}(\rho(x))}{\lambda_j(\rho(x))} \le C e^{-\mu \ell_S[x]}.
\]

\begin{definition}
Given a Borel Anosov representation $\rho$ into $\SL_m(\R)$ we define $\mu_j^{sing}$ (for each $j=1,\ldots, m$) to be the supremum over the set of all $\mu$ for which $\rho$ is $(C,\mu)$ $j$-dominated for some $C>0$.
We define each $\mu_j^{eig}$ similarly but for the eigenvalue domination constants.
\end{definition}
We would like to compare $\mu_j^{sing}$ and $\mu_j^{eig}$ for each $j$.
Clearly, if $\rho$ is $(C,\mu)$ $j$-dominated then $\rho$ is $(1,\mu)$ $j$-eigenvalue dominated. Hence we necessarily have that $\mu_j^{sing} \le \mu_j^{eig}$. Throughout the following we let $\kappa_j^-$ and $\kappa_j^+$ be
\[
\inf_{[x]\in \conj'} \frac{\log\lambda_{j}(\rho(x)) - \log\lambda_{j+1}(\rho(x))}{\ell_S[x]} \text{ and }
\sup_{[x] \in \conj'} \frac{\log\lambda_{j}(\rho(x)) - \log\lambda_{j+1}(\rho(x))}{\ell_S[x]}
\]
respectively for $j=1, \ldots, m-1$. Note that, $\kappa_j^{\pm} = \kappa_{d-j+1}^{\pm}$ for each $j$.

\begin{theorem}\label{thm.comp}
Suppose that $\rho : \G \to \SL_m(\R)$ is Borel Anosov and $S$ is a fixed finite generating set for $\G$. Then $
\mu_j^{sing} = \mu_j^{eig}= \kappa_j^- $
for each $j=1, \ldots, m-1$. Furthermore, for each $j = 1,\ldots m-1$ there exist $C_j^-, C_j^+ > 0$ such that for all $x \in \G$
\[
C_j^- e^{\kappa_j^- |x|_S} \le \frac{\sigma_{j}(\rho(x))}{\sigma_{j+1}(\rho(x))} \le C_j^+ e^{\kappa_j^+ |x|_S}.
\]
\end{theorem}

\begin{proof}
We begin by proving that $\mu_j^{sing} = \mu_j^{eig}.$
Given $\mu < \mu_j^{eig}$ we want to show that $\rho$ is $(C,\mu)$  $j$-dominated for some $C>0$.  

Take $x \in \G$. Then there exists $y$ satisfying the conditions in Proposition \ref{prop.AR2}. For such $y$ we have that
\[
\frac{\sigma_{j+1}(\rho(x))}{\sigma_{j}(\rho(x))} \le C_1 \frac{\lambda_{j+1}(\rho(y))}{\lambda_j(\rho(y))} \le C_2 e^{-\mu \ell_S[y]} \le C_3 e^{-\mu|y|_S} \le C_4 e^{-\mu|x|_S}
\]
for all $x\in\G$ where the constants $C_1, \ldots, C_4$ are independent of $x$. 
This shows the desired equality. The furthermore statement follows similarly to Theorem \ref{thm.Dildistancelike}.

\end{proof}
For Borel Anosov representations we can obtain a version of Proposition \ref{prop.arwm} that holds for all eigenvalue maps (not just the leading one). That is, we can obtain an optimal growth rate result that compares the eigenvalue maps and translation length (for a word metric) map at each index. We leave the details to the reader.


\section{Manhattan Geodesics}\label{sec.Manhattangeod}
In this section we construct the Manhattan geodesics in $\scrD_\G$ and prove Theorem \ref{thm.Mangeo}. Let $d, d_* \in \calD_{\Gamma}$ be two pseudo metrics, which we assume are \emph{not} roughly similar. Let $\theta=\theta_{d_*/d}: \mathbb{R} \rightarrow \mathbb{R}$ be the Manhattan curve for $d,d_*$ (which we introduced in Section \ref{sec.manhattan}) and write $h(d), h(d_\ast)$ for the exponential growth rates of $d,d_\ast$ respectively. By Theorem \ref{thm.manreg} we know that $\theta$ is strictly decreasing, convex, and continuously differentiable.

\begin{proposition}\label{prop.da}
     For any $t \in \mathbb{R}$ there exist a pseudo metric $d_{t}=d^\thet_{t} \in \calD_{\Gamma}$ and constant $C_{t} \geq 0$ such that
\begin{equation} \label{eq.da}
\left|d_{t}-\left(t d_*+\theta(t) d\right)\right| \leq C_{t}.
\end{equation}
The constants $C_{t}$ can be chosen so that
\begin{equation*}
C_{t}=\left\{\begin{array}{lll}
0 & \text { if } & 0 \leq t \leq h(d_*), \\
-\theta(t) C & \text { if } & t>h(d_*), \\
-t C & \text { if } & t<0,
\end{array}\right.
\end{equation*}
where $C$ is a constant depending only on $d$ and $d_*$.
\end{proposition}

\begin{remark}\label{rmk.busemann} For any $t\in \R$, the quasiconformal measures $\mu_{t, \theta(t)}$ from \cite[Cor.~2.10]{cantrell.tanaka.Man} are actually quasiconformal for $d_{t}$. That is, there exists $D\geq 1$ such that for every $x \in \Gamma$ and $\mu_{t, \theta(t)}$-a.e. $\xi \in \partial \Gamma$,
\begin{equation*}
D^{-1}e^{-\beta_{t}(x, o ; \xi)}\leq \frac{d x \mu_{t, \theta(t)}}{d \mu_{t, \theta(t)}}(\xi) \leq De^{-\beta_{t}(x, o ; \xi)},
\end{equation*}
where $\beta_t=\beta_{d_t}$ is the Busemann function for $d_t$.
\end{remark}

Let $D_{d_*,d}=\Dil(d_{*}, d), D_{d,d_*}=\Dil\left(d, d_{*}\right)$ and $\ell=\ell_d, \ell_\ast=\ell_{d_\ast}$, and for $t \in \mathbb{R}$ define $\widehat{d}_{t} :=t d_*+\theta(t) d$. 

\begin{lemma}\label{lemma.D0aposweak} 
    If $t >h(d_\ast)$ then $\theta(t) D_{d,d_\ast} + t \geq 0$.
\end{lemma}

\begin{proof}
    Suppose that $t+\theta(t)D_{d,d_\ast}<0$ for some $t>h(d_\ast)$. Then $\theta(t)<0$ and there exists $s>0$ small enough such that $\theta(t)+s<0$ and $\k:=-t-(\theta(t)+s)D_{d,d_\ast}$ is positive. We also consider $\ep>0$ small enough so that $\k+\ep(\theta(t)+s)>0$.
    
    By the definition of $D_{d,d_\ast}$ there exists $[x_\ep]\in \conj'$ such that $(D_{d,d_\ast}-\ep)\ell_\ast[x_\ep]\leq \ell[x_\ep]$. Since $(\theta(t)+s)$ is negative, this implies $(\theta(t)+s)\ep \ell_\ast[x_\ep]\leq (\theta(t)+s)(D_{d,d_\ast}\ell_\ast[x_\ep]-\ell[x_\ep])$. From this we get
    \begin{align*}
    \sum_{[x]\in \conj'}{e^{-t\ell_\ast[x]-(\theta(t)+s)\ell[x]}} & = \sum_{[x]\in \conj'}{e^{\k \ell_\ast[x]+(\theta(t)+s)(D_{d,d_\ast}\ell_\ast[x]-\ell[x])}}\\
    & \geq \sum_{n \geq 1}{e^{\k \ell_\ast[x_\ep^n]+(\theta(t)+s)(D_{d,d_\ast}\ell_\ast[x_\ep^n]-\ell[x_\ep^n])}}\\
    & =\sum_{n \geq 1}{e^{n(\k \ell_\ast[x_\ep]+(\theta(t)+s)(D_{d,d_\ast}\ell_\ast[x_\ep]-\ell[x_\ep]))}}\\
    & \geq \sum_{n \geq 1}{e^{n(\k+(\theta(t)+s)\ep)\ell_\ast[x_\ep]}},
    \end{align*}
    and our choice of $\ep$ implies that the series $\sum_{[x]\in \conj'}{e^{-t\ell_\ast[x]-(\theta(t)+s)\ell[x]}}$ diverges. As this contradicts the definition of $\theta(t)$ (see Remark \ref{rmk.conj=distMC}), the conclusion follows.
\end{proof}

\begin{lemma}\label{lem.defdt}
    There exists a constant $C'$ satisfying the following. If $t>h(d_\ast)$ then the expression
    \begin{equation}\label{eq.defd_t}
d_{t}(x, y):=\begin{cases}
\widehat{d}_{t}(x, y)+2|\theta(t)| C' & \text { if } x \neq y \\
0 & \text { otherwise }
\end{cases}
\end{equation}
defines a $\G$-invariant pseudo metric on $\G$ with Gromov product satisfying
\begin{equation}\label{eqd_tgroprod}
(x | y)_{o,d_t}  \leq t(x|y)_{o,d_\ast}+2|\theta(t)|C'
\end{equation}
for all $x,y\in \G$. In particular, $d_t$ is a hyperbolic distance-like function. Moreover, if $t+D_{d,d_\ast}\theta(t)>0$ then $d_t\in \calD_\G$. 
\end{lemma}

\begin{proof}
    We use the notation $(\cdot | \cdot)=(\cdot|\cdot)_{o,d}$, $(\cdot | \cdot)_{*}=(\cdot|\cdot)_{o,d_*}$, and $\widehat{(\cdot|\cdot)}_{t}=t(\cdot | \cdot)_{*}+\theta(t)(\cdot | \cdot)$ for all $t$. 
By Theorem \ref{thm.dilatta} there is a constant $C' \geq 0$ such that
\begin{equation*}
(x | y) \leq D_{d,d_*}(x | y)_{*}+C'
\end{equation*}
for all $x, y \in \Gamma$. Therefore, for all $x, y \in \Gamma$ and $t>h(d_\ast)$ we have
\begin{equation}\label{eq.grp}
\begin{aligned}
\widehat{(x | y)}_{t} &=t(x | y)_{*}+\theta(t)(x | y) \\
&=\left[t+D_{d,d_*} \theta(t)\right](x | y)_{*}-\theta(t)\left[D_{d,d_*}(x | y)_{*}-(x | y)\right] \\
& \geq\left[t+D_{d,d_*} \theta(t)\right](x | y)_{*}-|\theta(t)| C'\\
& \geq -|\theta(t)|C',
\end{aligned}
\end{equation}
where in the last inequality we used Lemma \ref{lemma.D0aposweak}.
Since $\widehat{d}_{t}$ is $\Gamma$-invariant we have $\widehat{d}_t(x,y)=\widehat{(x^{-1}y|x^{-1}y)}_t$ for all $x,y$, and so by \eqref{eq.grp} we get that the function $d_t$ defined in \eqref{eq.defd_t} is nonnegative and satisfies the triangle inequality. Therefore $d_t$ is a $\G$-invariant pseudo metric on $\G$. 

In addition, if $(\cdot | \cdot)_{t}$ denotes the Gromov product for $d_{t}$ based at the identity element, then for all $x, y \in \G$ we have
\begin{align*}
(x | y)_{t} & =\frac{d_t(x,o)+d_t(o,y)-d_t(x,y)}{2}\\ 
& \leq \frac{\widehat{d}_t(x,o)+\widehat{d}_t(o,y)-\widehat{d}_t(x,y)+4|\theta(t)|C'}{2}\\
&= \widehat{(x|y)}_t+2|\theta(t)|C'\\
& = t(x|y)_{\ast}+\theta(t)(x|y)+2|\theta(t)|C' \\
& \leq t(x|y)_{o,d_\ast}+2|\theta(t)|C',
\end{align*}
since $\theta(t)<0$. This proves \eqref{eqd_tgroprod} and implies that $d_t$ is a hyperbolic distance-like function.

Finally, if $t+D_{d,d_\ast}\theta(t)>0$  then from \eqref{eq.grp}
we see that $d_t$ is quasi-isometric to $d_\ast$, and hence to any word metric. By Remark \ref{rmk.hdlfhyp} this implies that $d_t$ is hyperbolic, so it belongs to $\calD_\G$. 
\end{proof}

\begin{proof}[Proof of Proposition \ref{prop.da}]
There are three cases to consider.

\emph{Case 1)} If $0 \leq t \leq h(d_*)$, then $\widehat{d}_{t} \in \calD_{\Gamma}$ by \cite[Lemma.~4.1]{oregon-reyes.ms}, so we take $d_t=\widehat{d}_t$ and $C_t=0$.

\emph{Case 2)} If $t>h(d_*)$, we let $d_t$ be given by Lemma \ref{lem.defdt}, which satisfies $|d_t-(td_\ast+\theta(t)d)|\leq C_t:=|\theta(t)|C$ for some $C$ depending only on $d$ and $d_\ast$. If $t+D_{d,d_\ast}\theta(t)>0$ then Lemma \ref{lem.defdt} implies that $d_t\in \calD_\G$. Therefore, by Lemma \ref{lemma.D0aposweak} it is enough to show that the equation $t+D_{d,d_\ast}\theta(t)=0$ has no solution for $t>h(d_\ast)$. 

To prove this last assertion, suppose by contradiction that $t_0+D_{d,d_\ast}\theta(t_0)=0$ for some minimal $t_0>h(d_\ast)$. Such a value $t_0$ would exist since $\theta(h(d_\ast))=0$ and $\theta$ is continuous. 
We claim that $h(d_t)=1$ for $h(d_\ast)<t<t_0$. Indeed, for any such $t$ we have $\theta'(t)\leq -D_{d,d_\ast}^{-1}<\theta(t)/t$ by the minimality of $t_0$ and the convexity of $\theta$, and hence the line $\{y=(\theta(t)/t) x\}$ is not tangent to $\theta$ at $(t,\theta(t))$. This implies that $(st,s\theta(t))$ belongs to the open convex set $\calC_{d_\ast/d}^M$ bounded by the graph of $\theta$ when $s>1$, and that $(st,s\theta(t))$ does not belong to the closure of $\calC_{d_\ast/d}^M$ when $s<1$. This gives us that the critical exponent of $s \mapsto \sum_{[x]\in \conj'}{e^{-s(t\ell_\ast[x]+\theta(t)\ell[x])}}$ is one, which equals $h(d_t)$ since $t\ell_\ast+\theta(t)\ell$ is the stable translation length function of $d_t$, see Remark \ref{rmk.conj=distMC}. 

In addition, from \eqref{eq.defd_t} it easily follows that $d_t$ pointwise converges to $t_0$ as $t$ tends to $t_0$. This fact combined with our claim and Lemma \ref{lem.proper} implies that $d_{t_0}$ is a proper hyperbolic distance-like function, which belongs to $\calD_\G$ by Corollary \ref{coro.calDiffproper}. In particular, there exists $\lam>0$ such that $\ell_{d_{t_0}}[x]=t_0\ell_\ast[x]+\theta(t_0)\ell[x] \geq \lam \ell_\ast[x]$ for all $[x]\in \conj$. 

To get our desired contradiction, we use the identity $t_0+D_{d,d_\ast}\theta(t_0)=0$, which implies
\[\lam \ell_\ast[x] \leq t_0\ell_\ast[x]+\theta(t_0)\ell[x]=-\theta(t_0)(D_{d,d_\ast}\ell_\ast[x]-\ell[x])=|\theta(t_0)|(D_{d,d_\ast}\ell_\ast[x]-\ell[x])\]
for all $[x]\in \conj$. This translates to 
\[\ell[x]\leq (D_{d,d_\ast}-|\theta(t_0)|^{-1}\lam)\ell_\ast[x]\]
for all $[x]\in \conj$, contradicting the definition of $D_{d,d_\ast}$ and concluding the proof of the proposition in this case.

\emph{Case 3)} If $t<0$, let $\psi$ be the Manhattan curve for $d_{*}, d$. Then $\psi=\theta^{-1}, t=\psi(s)$ for $s=\theta(t)>h(d)$, and $\widehat{d}_{t}=td_\ast+\theta(t)d=sd+\psi(s)d_\ast$, so the conclusion follows from Case 2 applied to the Manhattan curve $\psi$ at the value $s=\theta(t)$. In this case we obtain a constant of the form $C_t=|\psi(s)|C'=|t|C'$ for some $C'$ depending only on $d$ and $d_\ast$. 
\end{proof}

From the proof above and the fact that $\theta$ is positive on $(0,h(d_\ast))$ we deduce the following.

\begin{corollary}\label{coro.strictineq}
    For any $t>0$ we have $t+D_{d,d_\ast}\theta(t)>0$. Similarly, for all $t<h(d_\ast)$ we have $D_{d_\ast,d}t+\theta(t)>0$.
\end{corollary}

\begin{remark}\label{rmk.MCisnice}
The conclusion of the corollary above is subtle as it is false for arbitrary decreasing, strictly convex and continuously differentiable (even analytic) functions on $\R$. In \cite{cantrell.tanaka.2,cantrell.tanaka.Man}, it is shown that for some pairs of pseudo metrics in $\calD_\G$, their Manhattan curves can be recovered as pressure functions of H\"older potentials on mixing subshifts of finite type. We note that such pressure functions cannot be arbitrary, see \cite{kucherenko-quas}.
\end{remark}

\begin{remark}\label{rmk.h=1} (1) From the proof of Proposition \ref{prop.da} and the definition of $\theta$, we see that $h\left(d_{t}\right)=1$ for all $t\in\R$.

\noindent (2) We also note that for any $s < s_\ast$ we have that $s\theta(s_\ast) \neq s_\ast \theta(s)$. Note that if this equality holds then either $s < s_\ast < 0$ or $h(d_\ast) < s < s_\ast$. Suppose the latter is true, then there exists $\xi \in (s,s_\ast)$ with $\theta'(\xi) = \frac{\theta(s)}{s}$ and $ \theta(\xi) -\xi\theta'(\xi) < 0$, but then this implies that $\theta(t) + D_{d,d_\ast}^{-1}t \le \theta(\xi) -\xi\theta'(\xi) < 0$ contradicting Lemma \ref{lemma.D0aposweak}. The case that $s < s_\ast < 0$ can be proved similarly.
\end{remark}

\subsection{Explicit computations}
Let $d, d_*$ and $\thet$ be as above, and for $t\in \R$ let $\rho_{t}=\rho_t^\thet=\left[d_{t}\right]\in \scrD_\G$ be the metric structure induced by the pseudo metric $d_t$ from Proposition \ref{prop.da}. Let $\rho=[d]=\rho_0$ and $\rho_*=[d_*]=\rho_{h(d_*)}$.

For the rest of the section, we will use the notation $D_{s, t}=\Dil\left(d_{s}, d_{t}\right)$ and $\Delta_{s, t}=\Delta\left(\rho_{s}, \rho_{t}\right)$ for $s, t \in \mathbb{R}$, so that $\Delta_{s, t}=\log \left(D_{s, t} D_{t, s}\right)$, recall equation \eqref{eq.defDel}. Note that $d_0=h(d)d$ and $d_{h(d_*)}=h(d_*)d_*$. Theorem \ref{thm.Mangeo} will follow from the following estimates.

\begin{proposition}\label{prop.Mancomp} For any $t\in \mathbb{R}$ we have
\begin{align*}
    D_{0, t}&= \begin{cases}
        h(d)\left(t D_{d_*,d}+\theta(t)\right)^{-1} & \text {if $t<0$} \\
        h(d)\left(t D_{d,d_*}^{-1}+\theta(t)\right)^{-1} & \text {if $t >0$}
    \end{cases}\\
      D_{t, 0}&= \begin{cases}
        h(d)^{-1}(t D_{d,d_*}^{-1}+\theta(t)) & \text {if $t<0$} \\
        h(d)^{-1}(t D_{d_*,d}+\theta(t)) & \text {if $t >0$}
    \end{cases}\\
       D_{h(d_*),t} &= \begin{cases}
       h(d_*)\left(\theta(t) D_{d_*,d}^{-1}+t\right)^{-1} & \text {if $t<h(d_\ast)$} \\
        h(d_*)\left(\theta(t) D_{d,d_*}+t\right)^{-1} & \text {if $t >h(d_\ast)$}
    \end{cases}\\
       D_{t, h(d_*)}&= \begin{cases}
        h(d_*)^{-1}(\theta(t) D_{d,d_*}+t) & \text {if $t<h(d_\ast)$} \\
        h(d_*)^{-1}(\theta(t) D_{d_*,d}^{-1}+t) & \text {if $t >h(d_\ast)$}
    \end{cases}
\end{align*}
\noindent  and hence
\begin{equation*} \displaystyle e^{\Delta_{t, 0}}=\left\{\begin{array}{ll}\displaystyle \left(\frac{t D_{d,d_*}^{-1}+\theta(t)}{t D_{d_*,d}+\theta(t)}\right) & \text { if } \ t<0 \\ \displaystyle\left(\frac{t D_{d_*,d}+\theta(t)}{t D_{d,d_*}^{-1}+\theta(t)}\right) & \text { if } \ t>0 \end{array}\right.
\ 
 \displaystyle e^{\Delta_{t, h(d_*)}}=\left\{\begin{array}{ll}\displaystyle \left(\frac{\theta(t) D_{d,d_*}+t}{\theta(t) D_{d_*,d}^{-1}+t}\right) & \text { if } \ t<h(d_*) \\ \displaystyle \left(\frac{\theta(t) D_{d_*,d}^{-1}+t}{\theta(t) D_{d,d_*}+t}\right) & \text { if } \ t>h(d_*) .\end{array}\right.
\end{equation*}
\end{proposition}

We begin the proof of Proposition \ref{prop.Mancomp} with some lemmas. For two functions $f$, $g$ on a set $X$, the notation $f \lesssim g$ means that there is some $C \geq 0$ such that $f(x) \leq g(x)+C$ for all $x \in X$. We also write $f \eqsim g$ if $f \lesssim g$ and $g \lesssim f$.
\begin{lemma}\label{lem.compDa1}If $t>0$, then
\begin{itemize}
    \item[$i)$] $D_{0, t}=h(d)\left(t D_{d,d_*}^{-1}+\theta(t)\right)^{-1}$, and
\item[$ii)$] $D_{t, 0}=h(d)^{-1}(t D_{d_*,d}+\theta(t))$.
\end{itemize}
If $t<h(d_*)$, then
\begin{itemize}
    \item[$iii)$] $D_{t, h(d_*)}=h(d_*)^{-1}(\theta(t) D_{d, d_*}+t)$, and
\item[$iv)$] $D_{h(d_*), t}=h(d_*)\left(\theta(t) D_{d_*,d}^{-1}+t\right)^{-1}$.
\end{itemize}
\end{lemma} 

\begin{proof}
    Let $t>0$. We have $d \lesssim D_{d,d_*} d_{*}$, and hence $$d_{t} \eqsim t d_{*}+\theta(t) d \gtrsim\left(t D_{d,d_*}^{-1}+\theta(t)\right) d=h(d)^{-1}\left(t D_{d,d_*}^{-1}+\theta(t)\right) d_0.$$
By Corollary \ref{coro.strictineq}, $t D_{d,d_*}^{-1}+\theta(t) > 0$, and so $D_{0, t} \leq h(d)\left(t D_{d,d_*}^{-1}+\theta(t)\right)^{-1}$.
The reverse inequality of $i)$ is similar. From $d_{0} \leq D_{0, t} d_{t}$ we get
\begin{equation*}
d=h(d)^{-1}d_{0} \lesssim h(d)^{-1}[D_{0, t} t d_{*}+D_{0, t} \theta(t) d],
\end{equation*}
and hence
\begin{equation}\label{eq.D0a}
\left(1-h(d)^{-1}D_{0, t} \theta(t)\right) d \lesssim h(d)^{-1}D_{0, t}t d_{*}.
\end{equation}
The left hand side of \eqref{eq.D0a} is positive for $t \geq h(d_*)$, and for $0<t<h(d_*)$ we have
\begin{equation*}
h(d)^{-1}\left(1+\frac{t}{\theta(t)} D_{d,d_*}^{-1}\right) d_{0}=\left(1+\frac{t}{\theta(t)} D_{d,d_*}^{-1}\right) d \lesssim d+\frac{t}{\theta(t)} d_{*} \eqsim \theta(t)^{-1} d_{t} \text {, }
\end{equation*}
thus
\begin{equation*}
D_{0, t} \leq h(d)\theta(t)^{-1}\left(1+\frac{t}{\theta(t)} D_{d,d_*}^{-1}\right)^{-1}<h(d)\theta(t)^{-1}
\end{equation*}
and the left hand side of \eqref{eq.D0a} is positive for any $t>0$. This gives $D_{d,d_*} \leq h(d)^{-1}D_{0, t} t\left(1-h(d)^{-1}D_{0, t} \theta(t)\right)^{-1}$ or equivalently $D_{0, t} \geq h(d)\left(t D_{d,d_*}^{-1}+\theta(t)\right)^{-1}$.

We can prove $ii)$ in the same way, and identities $iii)$ and $iv)$ follow from $i)$ and $ii)$ applied to $\psi=\theta^{-1}$ and $s=\theta(t)$. Indeed, if $t<h(d_\ast)$ then $s>0$, and we note that $D_{h(d_*), t}=\Dil(d_{h(d_\ast)},d_t)=\Dil(d_0^\psi,d_s^\psi)$ and $D_{t, h(d_*)}=\Dil(d_t,d_{h(d_\ast)})=\Dil(d_s^\psi,d_0^\psi)$, for $d_0^\psi=\psi(0)d_\ast=h(d_\ast)d_\ast=d_{h(d_\ast)}$ and $d_s^\psi\eqsim sd+\psi(s)d_\ast=\theta(t)d+td_\ast\eqsim d_t$. Here $d_0^\psi,d_s^\psi$ are the pseudo metrics given by Proposition \ref{prop.da} applied to the curve $\psi$.
\end{proof}

Recall that we are assuming $[d]\neq [d_\ast]$. 

\begin{lemma}\label{lem.distinct}
If $r\neq t$ then the pseudo metrics $d_r$ and $d_t$ are not roughly similar.  
\end{lemma}

\begin{proof}
 Suppose that $r<t$ but $[d_r]=[d_t]$. By Remark \ref{rmk.h=1} (1) we have  $h(d_r)=h(d_t)=1$, and hence $d_r \eqsim d_r$ and $rd_\ast+\theta(r)d \eqsim td_\ast+\theta(t)d$. This implies $(t-r)d_\ast \eqsim (\theta(r)-\theta(t))d$, contradicting that $d$ and $d_\ast$ are not roughly similar.     
\end{proof}

\begin{lemma}\label{lem.reparam} Take real numbers $s < s_\ast$ and let $\ov\theta=\theta_{d_{s_\ast}/d_s}$ be the Manhattan curve for the pair $d_s,d_{s_\ast}$. Then for each $t\in \R$ we have that $\thet(\al(t))=\beta(t)$, where \begin{equation*}
    \al(t)=ts_*+\ov\thet(t)s \ \text{ and } \ \beta(t)=t\thet(s_*)+\ov\thet(t)\thet(s).
\end{equation*} Moreover, $\alpha$ and $\beta$ are bijections on $\R$.
\end{lemma}
\begin{proof} By the definition of $\ov\theta$, for all $t$ we get
    \begin{equation*}
\begin{aligned}
    d_t^{\ov\theta} \eqsim td_{s_\ast}+\ov\thet(t)d_s & \eqsim t(s_\ast d_*+\thet(s_\ast)d)+\ov\thet(t)(sd_*+\thet(s)d)\\
    & =[ts_*+\ov\thet(t)s]d_*+[t\thet(s_*)+\ov\thet(t)\thet(s)]d, 
\end{aligned}
\end{equation*}
where $d_t^{\ov\theta}$ is the pseudo metric given by Proposition \ref{prop.da} applied to $\ov\thet$. Since $h(d_t^{\ov\theta})=1$ by Remark \ref{rmk.h=1} (1), from the definition of $\theta$ we deduce that \begin{equation}\label{eq.alphabeta}
\theta(\alpha(t))=\theta(ts_\ast+\ov\theta(t)s)=t\theta(s_\ast)+\ov\theta(t)\theta(s)=\beta(t).
\end{equation}
    To prove the moreover statement we first observe that $\alpha$ and $\beta$ are $C^1$ and $\theta'(\alpha(t))\al'(t) = \beta'(t)$. We then note that the limits
    \[
    \lim_{t\to \infty}{\alpha'(t)}=\lim_{t\to\infty} \frac{\al(t)}{t} =s_\ast - D_{d_s,d_{s_\ast}}^{-1}s  \ \text{ and } \ \lim_{t\to \infty}{\beta'(t)}=\lim_{t\to\infty} \frac{\beta(t)}{t} = \thet(s_\ast) - D_{d_s,d_{s_\ast}}^{-1}\theta(s) 
    \]
    both exist by \cite[Cor.~3.3]{cantrell.tanaka.Man} and that we have similar expressions when $t \to -\infty$.
    Since $\theta'(\alpha(t))\in [-D_{d_\ast,d},-D_{d,d_\ast}^{-1}]$ for all $t$, it follows from the expression $\theta'(\alpha(t))\al'(t) = \beta'(t)$ that when $t\to \infty$, both $|\al(t)| \to \infty$ and $|\beta(t)| \to \infty$ unless $s_\ast - D_{d_s,d_{s_\ast}}^{-1}s = \thet(s_\ast) - D_{d_s,d_{s_\ast}}^{-1}\theta(s)=0 $. However this equality would imply $s \theta(s_\ast) = s_\ast \theta(s)$ contradicting Remark \ref{rmk.h=1} (2). A similar argument shows that as $t\to -\infty$ both $|\al(t)| \to \infty$ and $|\beta(t)| \to \infty$.
    
    Therefore, to conclude the proof it suffices to show that the derivatives of $\al$ and $\beta$ are never $0$ (as they are then strictly monotone and unbounded on $(-\infty,0)$ and $(0,\infty)$). Note that $\al'(t) = 0$ if and only if $\beta'(t) = 0$ and if there exists $t_0 \in \R$ with $\al'(t_0) = \beta'(t_0) = 0$ then $s_\ast\ov{\theta}'(t_0) + s = \theta(s_\ast)\ov{\theta}'(t_0) +\theta(s) = 0$ from which it follows that $ s_\ast \ov\theta(s) = s\ov\theta(s_\ast)$. Again this is impossible due to Remark \ref{rmk.h=1} (2) and the proof is complete.
\end{proof}

\begin{lemma}\label{lem.Dgeod}
    If $r<s<t$, then
\begin{equation*}
D_{r,t}=D_{r, s} \cdot D_{s, t} \text { and } D_{t, r}=D_{t, s} \cdot D_{s, r} \text {. }
\end{equation*}
\end{lemma}
\begin{proof}
    For the case when $r=0$ and $t=h(d_*)$, the conclusion follows easily from Lemma \ref{lem.compDa1}.  For the general case, let $\psi$ be the Manhattan curve for $d_{r}, d_{t}$, which belong to $\calD_\G$ by Proposition \ref{prop.da}. By Lemma \ref{lem.distinct} we have $[d_r]\neq [d_t]$ and by Lemma \ref{lem.reparam} we can reparameterise $\psi$ using the bijection $\alpha(a):=a t+\psi(a) r$ so that $d_{0}^{\psi} \eqsim d_{r}$, $d_{1}^{\psi} \eqsim d_{t}$ and in fact $d_{a}^{\psi} \eqsim d_{\alpha(a)}^{\theta}$ for all $a$. Since $h(d_t)=1$, the general case then follows from the first case applied to $\psi$ and the value $0<\tilde{s}<1$ satisfying $\alpha(\tilde{s})=s$.
\end{proof}

\begin{lemma}\label{lem.compDa2} If $t<0$, then
\begin{itemize}
\item[$i)$] $D_{0, t}=h(d)\left(t D_{d_*,d}+\theta(t)\right)^{-1}$, and
\item[$ii)$] $D_{t, 0}=h(d)^{-1}(t D_{d,d_*}^{-1}+\theta(t))$.
\end{itemize}
Also, if $t>h(d_*)$, then
\begin{itemize}
\item[$iii)$] $D_{t, h(d_*)}=h(d_*)^{-1}(\theta(t) D_{d_*,d}^{-1}+t)$, and
\item[$iv)$] $D_{h(d_*), t}=h(d_*)\left(\theta(t) D_{d,d_*}+t\right)^{-1}$.
\end{itemize}
\end{lemma}
\begin{proof} From Lemmas \ref{lem.compDa1} and \ref{lem.Dgeod}, for $t<0$ we have
\begin{equation*}
D_{0, t}=D_{h(d_*), t} / D_{h(d_*),0}=h(d)\left(t D_{d_*,d}+\theta(t)\right)^{-1}
\end{equation*}
and
\begin{equation*}
D_{t,0}=D_{t,h(d_*)} / D_{0,h(d_*)}=h(d)^{-1}(t D_{d,d_*}^{-1}+\theta(t)).
\end{equation*}
Identities $iii)$ and $iv)$ are deduced in an analogous way.
\end{proof}

\begin{proof}[Proof of Proposition \ref{prop.Mancomp}]
Lemmas \ref{lem.compDa1} and \ref{lem.compDa2} imply the result, since from them we can already verify the formulas for $\Delta_{t, 0}$ and $\Delta_{t, h(d_*)}$.   
\end{proof}

\begin{proof}[Proof of Theorem \ref{thm.Mangeo}]
For each $t\in \R$, let $\rho_t=[d_t]$ as above, for which statement $i)$ holds by definition and statement $ii)$ follows from Lemma \ref{lem.Dgeod}. 
For statement $iii)$ we compute
\small\begin{equation*}
\lim_{t\to \infty}{e^{\Delta_{t, h(d_*)}}}=\lim_{t \to \infty}\left(\frac{\theta(t) D_{d_*,d}^{-1}+t}{\theta(t) D_{d,d_*}+t}\right) =\lim_{t \rightarrow \infty}  \left(\frac{1+\frac{\theta(t)}{t} D_{d_*,d}^{-1}}{1+\frac{\theta(t)}{t} D_{d,d_*}}\right)
= \infty,
\end{equation*}
\normalsize
where we used $\lim_{t \rightarrow \infty} \frac{\thet(t)}{t}=-D_{d,d_*}^{-1}$ by \cite[Cor.~3.3]{cantrell.tanaka.Man}. Similarly $\lim _{t \rightarrow-\infty}{\Del_{t, 0}}=\infty$.

Finally, note that $\Delta_{0, t}$ and $\Delta_{h(d_*), t}$ are continuous functions on $t$, so that $\lim _{s \rightarrow t} \Delta_{s, t}=0$ for any $t$. Since $[d]\neq [d_*]$, we have $D_{d,d_*}\cdot D_{d_*,d}>1$ and from Proposition \ref{prop.Mancomp} we deduce that $\Delta_{s, t}>0$ for $s \neq t$, and hence $\rho_\bullet^\thet$ is continuous and injective.
\end{proof}

\begin{remark}\label{rem.asympconst}
From Proposition \ref{prop.Mancomp} we deduce that $$0<t D_{d_*,d}+\theta(t) \leq h(d) \text{ for }t<0, $$
and $$\quad 0<t D_{d,d_*}^{-1}+\theta(t) \leq h(d_*)D_{d,d_*}^{-1} \text{ for }t>h(d_*).$$
Therefore, $$\theta(t)=-t D_{d_*,d}+O(1) \hspace{2mm}\text{  and } \hspace{2mm}\theta(-t)=t D_{d,d_*}^{-1}+O(1) \hspace{2mm}\text{ as }t \rightarrow-\infty,$$
which generalizes \cite[Prop.~4.22]{cantrell.tanaka.Man} to arbitrary pairs of metrics $d, d_{*} \in \calD_{\Gamma}$.
\end{remark}

\subsection{The geodesic bicombing}
As we mentioned in the introduction, by appropriately reparametrizing the curves $\rho_\bullet^{d_*/d}$ given by Theorem \ref{thm.Mangeo} we can produce a geodesic bicombing on $\scrD_\G$ by bi-infinite geodesics.

\begin{definition}\label{def.Manhgeo1} For two distinct metric structures $\rho=[d],\rho_*=[d_*]$ in $\scrD_\G$, the \emph{Manhattan geodesic} of the pair $\rho,\rho_*$ is the map $\sigma^{\rho_*/\rho}_\bullet: \R \ra \scrD_\G$ given by the arc-length reparametrization of the map $\rho_\bullet^{d_*/d}$ such that $\sigma^{\rho_*/\rho}_0=\rho$ and $\sigma^{\rho_*/\rho}_{\Del(\rho,\rho_*)}=\rho_*$.
\end{definition}

More precisely, if $\rho=[d]$ and $\rho_*=[d_*]$, then $\sigma_t^{\rho_*/\rho}$ equals $\rho_{\gam(t)}^{d_*/d}$, where $\gamma(t)$ is the unique number such that \begin{equation}\label{eq.defMgeod}
\Del(\rho,\rho_{\gam(t)})=t \hspace{2mm}\text{ and }\hspace{2mm}t\cdot \gam(t)\geq 0.
\end{equation}

Manhattan geodesics are well-defined, since for $\rho$ and $\rho_*$ as in the preceding definition, the image $\rho^{d_*/d}(\R)\subset \scrD_\G$ and the orientation of the curve $\rho^{d_*/d}_\bullet$ do not depend on the representatives $d$ and $d_*$. We end this section by proving some properties of the geodesic bicombing consisting of the Manhattan geodesics.

\begin{theorem}\label{thm.bicombing}
The geodesic bicombing $(\rho,\rho_*) \mapsto \sigma^{\rho_*/\rho}_\bullet$ satisfies the following.
\begin{itemize}
    \item \emph{Continuity}: if $\rho^n \to \rho$, $\rho^n_* \to \rho_*$ and $\rho \neq \rho_*$ in $\scrD_\G$, then $\sigma^{\rho_*^n/\rho^n}_\bullet$ converges to $\sigma_\bullet^{\rho_*/\rho}$ uniformly on compact subsets of $\R$. 
    \item \emph{$\Out(\G)$-invariance}: $\phi \circ \sigma_\bullet^{\rho_*/\rho}=\sigma_\bullet^{\phi(\rho_*)/\phi (\rho)}$ for any $\phi\in \Out(\G)$ and $\rho\neq \rho_*$.
    \item \emph{Consistency}: if $\rho\neq \rho_*$ and $\tau=\sigma_s^{\rho_*/\rho}, \tau_*=\sigma_{s_*}^{\rho_*/\rho}$  for $s\neq s_*$, then 
    \[
\sigma^{\tau_*/\tau}_t=\sigma^{\rho_*/\rho}_{T(s,s_\ast,t)} \  \text{ where } \ T(s, s_\ast, t) = t\cdot \left(\frac{s_*-s}{|s_\ast-s|}\right)+s \hspace{2mm} \text{ for each } t\in \R.
    \]
\end{itemize}
\end{theorem}
 
In the last point of this theorem, consistency refers to the following property. Suppose that $\sigma_\bullet^{\rho_\ast/\rho}$ is the bi-infinite geodesic joining $\rho$ and $\rho_\ast$ that we constructed above. Then,  if we pick two metric structures $\tau, \tau_\ast$ lying on $\sigma_\bullet^{\rho_\ast/\rho}$ and construct the bi-inifinite  geodesic joining $\tau, \tau_\ast$, then we obtain the same bi-inifinite geodesic $\sigma_\bullet^{\rho_\ast/\rho}$ up to reparameterisation.

\begin{proof}
    The bicombing satisfies $\Out(\G)$-invariance, since for any $d,d_*\in \calD_\G$ and $\phi \in\Aut(\G)$ we have $\rho_\bullet^{\phi(d_*)/\phi(d)}=\phi \circ \rho_\bullet^{d_*/d}$.

    To prove continuity, consider sequences $\rho^n=[d^n]$ and $\rho^n_*=[d^n_*]$ in $\scrD_\G$ converging to $\rho=[d]$ and $\rho_*=[d_*]$ as $n$ tends to infinity, respectively. We can assume that $d,d_*,d^n$, and $d^n_*$ have exponential growth rates equal to 1 for all $n$. Under this assumption, if we let $\thet_n=\thet_{d^n_*/d^n}$ and $\thet=\thet_{d_*/d}$ then $\thet_n$ converges to $\thet$ uniformly on compact subsets of $\R$, see the proof of \cite[Thm.~1.9]{oregon-reyes.ms}. From this we deduce that if $\rho^n_\bullet=\rho_\bullet^{d^n_*/d^n}$, then $\rho^n_\bullet$ converges to $\rho_\bullet$ uniformly on compact subsets of $\R$. Continuity follows from this property and \eqref{eq.defMgeod}.
    
    Finally, consistency follows from the fact that if $\rho=[d] \neq \rho_*=[d_*]\in \scrD_\G$ and $\tau \neq \tau_*\in \sigma^{\rho_*/\rho}_\bullet$, then the curves $\sigma^{\tau_*/\tau}_\bullet$ and $\sigma^{\rho_*/\rho}_\bullet$ have the same image in $\scrD_\G$. To prove this fact suppose that $\tau=[d_s]=\rho_s^{d_*/d}$ and $\tau_*=[d_{s_\ast}]=\rho_{s_*}^{d_*/d}$ for some $s\neq s_*$.
    Then, as shown in Lemma \ref{lem.reparam} there is a bijection $\alpha(t)$ such that $\sigma^{\tau_\ast/\tau}_t = \sigma^{\rho_\ast/\rho}_{\alpha(t)}$ for all $t \in \R$.
This concludes the proof of the fact, and hence the theorem.
\end{proof}


\section{The Manhattan Boundary}\label{sec:manhattanboundary}
In this section, we discuss the Manhattan boundary of $\scrD_\G$ and prove theorem \ref{thm.Mboundaryvisible}, which characterizes its elements as the limits at infinity of Manhattan geodesics. 

As in the previous section, let $d,d_*\in \calD_\G$ be a pair of non-roughly similar pseudo metrics, let $\thet$ be its Manhattan curve, and $t\mapsto \rho_t=[d_t]$ be the reparametrization of the Manhattan geodesic for $\rho=[d],\rho_*=[d_*]$ defined in terms of $\thet$. We keep the notation $D_{d,d_*}=\Dil(d,d_*)$ and $D_{d_*,d}=\Dil(d_*,d)$.

\begin{proposition}\label{prop.d+-inf}
There are left-invariant pseudo metrics $d_{-\infty}=d_{-\infty}^\thet$ and $d_{\infty}=d_\infty^\thet$ on $\G$ and a constant $C\geq 0$ such that 
\begin{equation}\label{eq.d+inf}
    |d_\infty-(D_{d,d_*}d_*-d)|\leq C \ \text{ and } \  
    |d_{-\infty}-(D_{d_*,d}d-d_*)|\leq C.
\end{equation}
The pseudo metrics $d_{\infty}$ and $d_{-\infty}$ satisfy:
\begin{enumerate}
   \item
    $\displaystyle\ell_\infty:=\ell_{d_{\infty}}=\lim_{t \to \infty}{\frac{1}{-\thet(t)}\ell_{d_t}} \hspace{2mm}\text{ and }\hspace{2mm}\displaystyle\ell_{-\infty}:=\ell_{d_{-\infty}}=\lim_{t \to -\infty}{\frac{1}{-t}\ell_{d_t}}; \text{ and,}$ 
    \item they both belong to $\partial_M \calD_\G$.
\end{enumerate}
\end{proposition}

\begin{proof}
By Theorem \ref{thm.dilatta}, there is a constant $C'\geq 0$ such that 
\begin{equation}
    D_{d,d_*}^{-1}(x|y)_{o,d}-D_{d,d_*}^{-1}C'\leq (x|y)_{o,d_*} \leq D_{d_*,d} (x|y)_{o,d}+C'
\end{equation}
for all $x,y\in \G$. Therefore, as in the proof of Proposition \ref{prop.da}, the functions
\begin{equation*}
d_{\infty}(x, y):=\begin{cases}
D_{d,d_*}d_*(x,y)-d(x,y)+2C' & \text { if } x \neq y \\
0 & \text { otherwise }
\end{cases}
\end{equation*}
and
\begin{equation*}
d_{-\infty}(x, y):=\begin{cases}
D_{d_*,d}d(x,y)-d_*(x,y)+2C' & \text { if } x \neq y \\
0 & \text { otherwise }
\end{cases}
\end{equation*}
define left-invariant pseudo metrics on $\G$ verifying \eqref{eq.d+inf} with $C=2C'$.

Now we check the desired properties for $d_{-\infty}$ and $d_\infty$.

First, we compute
$$\lim_{t \to \infty}{\frac{1}{-\thet(t)}\ell_{d_t}}=\lim_{t \to \infty}{\frac{(t\ell_{d_*}+\theta(t)\ell_{d})}{-\thet(t)}}=D_{d,d_*}\ell_{d_*}-\ell_{d},$$
where we use $\lim_{t \to \infty}\frac{t}{-\thet(t)}=D_{d,d_*}$ \cite[Cor.~3.3]{cantrell.tanaka.Man}. Similarly, the identity $\lim_{t \to -\infty}\frac{\thet(t)}{-t}=D_{d_*,d}$ gives the analogous result for $\ell_{-\infty}$. The functions $\ell_\infty$ and $\ell_{-\infty}$ are non-constant since $d$ and $d_*$ are not roughly isometric, and hence $d_\infty$ and $d_{-\infty}$ satisfy (1).

In addition, we have
\begin{equation}\label{eq.groineqd+-inf}
(x|y)_{o,d_{\infty}}\leq D_{d,d_*}(x|y)_{o,d_*}+C \hspace{2mm}\text{ and }\hspace{2mm}(x|y)_{o,d_{-\infty}}\leq D_{d_*,d}(x|y)_{o,d}+C,
\end{equation}
so $d_{\infty}$ and $d_{-\infty}$ belong to $\ov\calD_\G$.

Finally, by the definition of $D_{d,d_*}$ there is a sequence $x_n \in \G$ such that 
\[
\ell_{d}[x_n]\geq \left(D_{d,d_*}-\frac{1}{n}\right)\ell_{d_*}[x_n]>0
\]
for all $n$, and hence 
$$\ell_{\infty}[x_n]=D_{d_*,d}\ell_{d_*}[x_n]-\ell_{d}[x_n]\leq \frac{1}{n}
\ell_{d_*}[x_n].$$

This implies that $d_\infty$ is not quasi-isometric to $d_1$. Similarly, $d_{-\infty}$ is not quasi-isometric to $d_0$, which proves that $d_{-\infty},d_{\infty}\in \partial_M \calD_\G$, and hence (2).
\end{proof}

Recall from Definition \ref{def.limitsofManhattan} in the introduction that if $\sigma$ is the Manhattan geodesic for the pair $\rho=[d],\rho_*=[d_*]$ with $\rho\neq \rho_*$, the limit at infinity of $\sigma$ is the boundary metric structure $\sigma_\infty=\sigma^{\rho_*/\rho}_\infty$ represented by pseudo metrics roughly similar to $\Dil(d,d_*)d_*-d$. The limit at negative infinity of $\sigma$, denoted $\sigma_{-\infty}=\sigma_{-\infty}^{\rho_*/\rho}$,  is defined similarly.

\begin{theorem}\label{thm.Mboundaryvisible}
    For any $\rho\in \scrD_\G$ and $\rho_\infty\in \partial_M\scrD_\G$, there exists some $\rho_*\in \scrD_\G$ such that $\rho_\infty=\sigma_\infty^{\rho_*/\rho}$. Moreover, if $\rho'_*\in \scrD_\G$ satisfies $\rho_\infty=\sigma_\infty^{\rho'_*/\rho}$ then $\rho_*'\in \sigma^{\rho_*/\rho}(0,\infty)$.
\end{theorem}

We need a preliminary lemma. 

\begin{lemma}\label{lem.int+bounr=int}
    If $d\in \calD_{\G}$ and $d_\infty\in \partial_M \calD_\G$, then $d+d_\infty \in \calD_\G$.
\end{lemma}
\begin{proof}
Clearly $d+d_\infty$ is a left-invariant pseudo metric on $\G$. It also satisfies \eqref{eq.ineqdefboundarymetric} for some $\lam>0$ and $d_0\in \calD_\G$, since $d$ and $d_\infty$ do. Therefore, $d+d_\infty$ is roughly geodesic by Lemma \ref{lem.RGforHDLF}. It is also quasi-isometric to a word metric since it is bounded below by $d\in \calD_\G$, so it is hyperbolic by Remark \ref{rmk.hdlfhyp}. This concludes the proof of the lemma.
\end{proof}

For the rest of the subsection we use the notation `$\lesssim, \eqsim$' introduced right before Lemma \ref{lem.compDa1}.

\begin{proof}[Proof of Theorem \ref{thm.Mboundaryvisible}]
Let $\rho_\infty=[d_\infty]$ and $\rho=[d]$. Define $d_*:=d+d_\infty$, which is a pseudo metric in $\calD_{\G}$ by Lemma \ref{lem.int+bounr=int}.
Since $d_\infty \in \partial_M \calD_\G$, we have that $d$ and $d_*$ are not roughly similar. 

We claim that $\Dil(d,d_*)=1$. Indeed, since $d=d_*-d_\infty\leq d_*,$ we get $\Dil(d,d_*)\leq 1$. In addition, by Theorem \ref{thm.dilatta} there is some $C\geq 0$ such that 
$$(1-\Dil(d,d_*))d_*\leq d_*-d+C=d_\infty+C,$$ 
and since $d_\infty$ is not quasi-isometric to a word metric, we get $\Dil(d,d_*)\geq 1$.

Therefore, by our claim we deduce $d_\infty=\Dil(d,d_*)d_*-d$, and $\rho_\infty=\sigma^{\rho_*/\rho}_{\infty}$ for $\rho_*=[d_*]$.

Finally, suppose that $\rho_\infty=\sigma_\infty^{\tilde{\rho}_*/\rho}$ for some $\tilde\rho_*=[\tilde{d}_*]$. Then there exists $\lam>0$ such that 
$$\lam(d_*-d) =\lam d_\infty \eqsim \Dil(d,\tilde{d}_*)\tilde{d}_*-d.$$
We get
$$h(\tilde{d}_*)\tilde{d}_* \eqsim h(\tilde{d}_*)\lam\Dil(d,\tilde{d}_*)^{-1}d_*+h(\tilde{d}_*)(1-\lam)\Dil(d,\tilde{d}_*)^{-1}d,$$
and we conclude $\tilde{\rho}_*=\rho^{d_*/d}_t$ for $t=h(\tilde{d}_*)\lam \Dil(d,\tilde{d}_*)^{-1}>0$, so that $\tilde{\rho}_*\in \sigma^{\rho_*/\rho}(0,\infty)$.
\end{proof}

From the proof of Theorem \ref{thm.Mboundaryvisible} we deduce:
\begin{corollary}\label{coro.dinftisdifference}
    For any $d_\infty\in \partial_M\calD_\G$ and $d\in \calD_\G$ there exists $d_*\in \calD_\G$ such that 
    $$d_\infty=\Dil(d,d_*)d_*-d.$$
\end{corollary}


We end this section by characterising when two boundary metric structures are the positive and negative limits of a Manhattan geodesic.
\begin{definition}\label{def.transverse}
    Two boundary metric structures $\rho,\rho_* \in \partial_M\scrD_\G$ are \emph{transverse} if for some (any) $d\in \rho$ and $d_*\in \rho_*$ we have $d+d_*\in \calD_\G$. 
\end{definition}
The equivalence of `some' and `any' in the above definition follows from the fact that if $d \in \rho, d_\ast \in \rho_\ast$ with $d+d_\ast \in \Dc_\G$ then for any $a > b >0$ we have that $ad+bd_\ast = b(d+ d_\ast) + (a-b)d$ (which belongs to $\Dc_\G$). 
\begin{proposition}\label{prop.chartransverse}
    The boundary metric structures $\rho,\rho_*\in \partial_M \scrD_\G$ are transverse if and only if there is a Manhattan geodesic $\sigma_\bullet$ such that $\rho=\sigma_{-\infty}$ and $\rho_*=\sigma_{\infty}$.
\end{proposition}
\begin{proof}
    Suppose first that $\rho=\sigma_{-\infty}$ and $\rho_*=\sigma_\infty$ for $\sigma_\bullet=\sigma_\bullet^{\tau_*/\tau}$ the Manhattan geodesic for $\tau=[d],\tau_*=[d_*]$, so that $d$ and $d_*$ are not roughly similar and have exponential growth rate 1. We consider $d_{-\infty}\in \rho$ and $d_{\infty}\in \rho_*$, which up to rescaling, we can assume satisfy
    $$d_\infty \eqsim \Dil(d,d_*)d_*-d \text{ and }d_{-\infty} \eqsim \Dil(d_*,d)d-d_*.$$ 
    In particular, we have
    $$d_\infty+d_{-\infty} \eqsim (\Dil(d_*,d)-1)d+(\Dil(d,d_*)-1)d_*,$$
    and this last pseudo metric is quasi-isometric to pseudo metrics in $\calD_\G$ since $\Dil(d,d_*)\cdot\Dil(d_*,d)>1$ and $\Dil(d,d_\ast),\Dil(d_\ast,d)\geq 1$ \cite[Lem.~3.6]{oregon-reyes.ms}. By Remark \ref{rmk.hdlfhyp}, this implies that $d_\infty+d_{-\infty}\in \calD_\G$, and hence $\rho$ and $\rho_*$ are transverse.

    For the reverse implication, suppose that $\rho=[d_{-\infty}]$ and $\rho_*=[d_\infty]$ in $\partial_M \scrD_\G$ are transverse. By assumption, $ad_{\infty}+bd_{-\infty}\in \calD_\G$ for any $a,b>0$, and in particular the pseudo metrics
    \begin{equation}\label{eq.d01fromdd'}
        d:=d_\infty+2d_{-\infty} \text{ and } d_*:=2d_\infty+d_{-\infty},
    \end{equation}
    belong to $\calD_\G$. We have that $d$ and $d_*$ are not roughly isometric, since otherwise we would get $d \eqsim \lam d_*$ for some $\lam>0$ and hence $(1-2\lam)d_\infty \eqsim (\lam-2)d_{-\infty},$ contradicting $d_\infty+d_{-\infty}\in \calD_\G$. 

    From \eqref{eq.d01fromdd'} we get
     $$d=d_\infty+2d_{-\infty}=d_\infty+2(d_*-2d_\infty)=2d_*-3d_\infty,$$
     and $2d_*-d=3d_\infty\geq 0$. This gives $d\leq 2d_*$ and hence $\Dil(d,d_*)\leq 2$. But, if $2=\Dil(d,d_*)+\al$ for some $\al>0$, then we would have
     $$3d_\infty=2d_*-d=(\Dil(d,d_*)+\al)d_*-d=\al d_*+(\Dil(d,d_*)d_*-d) \gtrsim \al d_*,$$
     contradicting that $d_\infty\in \partial_M \calD_\G$. We obtain $\Dil(d,d_*)=2$, and by the same argument we deduce $\Dil(d_*,d)=2$.

     To conclude the result, if $\sigma_\bullet=\sigma_\bullet^{\tau_*/\tau}$ is the Manhattan geodesic for $\tau=[d]$ and $\tau_*=[d_*]$, then by Proposition \ref{prop.d+-inf} there are pseudo metrics $d^\sigma_{\pm\infty}\in \partial_M\calD_\G$ representing $\sigma_{\pm \infty}$ and satisfying \begin{equation*}
         d^\sigma_\infty \eqsim \Dil(d,d_*)d_*-d=2d_*-d=3d_\infty,
     \end{equation*}
     \begin{equation*}
         d^\sigma_{-\infty} \eqsim \Dil(d_*,d)d-d_*=2d-d_*=3d_{-\infty}.
     \end{equation*}
    We get that $d^\sigma_\infty$ is roughly similar to $d_\infty$ and $d^\sigma_{-\infty}$ is roughly similar to $d_{-\infty}$, so that $\rho=\sigma_{-\infty}$ and $\rho_*=\sigma_{\infty}$, as desired.
\end{proof}
\begin{corollary}\label{coro.transverse}
    If $[d_\infty],[d_{-\infty}]\in \partial_M\scrD_\G$ are transverse, then there exist $d,d_*\in \calD_\G$ such that 
    $$d_\infty \eqsim \Dil(d,d_*)d_*-d\ \ \ \text{ and } \ \ \ d_{-\infty} \eqsim \Dil(d_*,d)d-d_*.$$
\end{corollary}


\section{Examples of boundary pseudo metrics}\label{sec:examples}

In this section, we provide concrete examples of pseudo metrics in $\ov\calD_\G$ and $\partial_M\calD_\G$, which include the ones from Theorem \ref{thm.examples}. These pseudo metrics will be induced from actions of $\G$ on hyperbolic spaces, according to the following definition.

\begin{definition}\label{def.orbitpseudo}
    Let $(X,d_X)$ be a pseudo metric space endowed with an isometric action of $\G$. By a \emph{orbit pseudo metric} induced by the action of $\G$ on $X$, we mean any pseudo metric on $\G$ of the form $d_X^p(x,y)=d_X(xp,yp)$ for $x,y\in \G$, where $p\in X$ is a base point. The rough similarity class $\rho_X=[d_X^p]$ is independent of the point $p$, and when appropriate, we say it is the (boundary) metric structure \emph{induced} by the action of $\G$ on $(X,d_X)$.
 \end{definition}

\subsection{Useful criteria} In general, verifying condition \eqref{eq.ineqdefboundarymetric} in Definition \ref{def.Manhattanboundarymetrics} is not at all direct. Instead, we will rely on the following criterion, for which similar instances have appeared in the literature.

\begin{lemma}\label{lem.BCCchar}
   A left-invariant pseudo metric $d$ on $\G$ belongs to $\ov\calD_\G$ if and only if $\ell_d$ is non-identically zero and additionally:
    \begin{enumerate}
        \item[i)]  $d$ is $\al$-roughly geodesic for some $\al\geq 0$; and,
        \item[ii)] if $d_0\in \calD_\G$ is $\al_0$-roughly geodesic, then there exists some $C\geq 0$ such that if $\gam \subset \G$ is an $(\al_0,d_0)$-rough geodesic with endpoints $x,y$, then $\gam$ is $C$-Hausdorff close in $d$ to an $(\al,d)$-rough geodesic with endpoints $x,y$.
    \end{enumerate}
\end{lemma}
This lemma follows immediately from the next statement.
\begin{lemma}\label{lem.charGro=BBT} Let $d_0$ and $d$ be left-invariant pseudo metrics on $\G$, and assume $d_{0}\in \calD_\G$ is $\delta_{0}$-hyperbolic and $\alpha_{0}$-roughly geodesic. Then the following conditions are equivalent:
\begin{enumerate}
    \item There exists $\lambda>0$ such that
$$
(x | y)_{w, d} \leq \lambda (x | y)_{w, d_{0}}+\lambda
\hspace{ 2mm } \text{ for all }x, y, w \in \G.$$
    \item $d$ is $\alpha$-roughly geodesic for some $\alpha$, and there is $C>0$ such that for all $x, y \in \G$ the following holds:
 if $\gamma$ is an $(\alpha_{0}, d_{0})$-rough geodesic with endpoints $x, y$, and $\beta$ is an $(\alpha, d)$-rough geodesic with endpoints $x, y$, then $\beta$ and $\gamma$ are $C$-Hausdorff close in the pseudo metric $d$.
 \end{enumerate}
 Furthermore, if either of the these statements hold then $d$ is hyperbolic.
\end{lemma}

\begin{proof}
Suppose first that $d$ satisfies (1), so that it is $\alpha$-roughly geodesic by Lemma \ref{lem.RGforHDLF}. To prove (2), let $\gamma$ and $\beta$ be $\left(\alpha_{0}, d_{0}\right)$ and $(\alpha, d)$-rough geodesics respectively, with endpoints $x, y \in \G$. We claim that these geodesics are $C$-Hausdorff close in $d$ for some $C$ independent of $\beta$ and $\gam$.

To this end, let $u\in \gam$, so that $(x|y)_{u,d_0}\leq 3\al_0/2$ and by (1) we get $(x|y)_{u,d}\leq 3\lam \al_0/2+\lam$. By applying of Lemma \ref{lemma:seqRG} to the sequence $a_i=d(x,v_i)$ for $v_i$ the elements in $\beta$, we can find some $v\in \beta$ such that \begin{equation}\label{eq.11}|d(x,u)-d(x,v)|\leq 3\lam \al_0+2\lam+\al+1,
\end{equation}
and hence 
\begin{equation}\label{eq.12}|d(u,y)-d(v,y)|\leq 6\lam \al_0+4\lam+4\al+1.
\end{equation}
Also, by $\del_0$ hyperbolicity of $d_0$ we have
\begin{equation}\label{eq.groprod}\min\{(x|v)_{u,d},(v|y)_{u,d}\}\leq \lam (x|y)_{u,d_0}+\lam \del_0 +\lam\leq 3\lam \al_0/2+\lam \del_0+\lam.
\end{equation}
Independently on what Gromov product achieves the minimum on the left-hand side of \eqref{eq.groprod}, by applying inequalities \eqref{eq.11} or \eqref{eq.12}  we end up concluding that $d(u,v)\leq \lam (9\al_0+2\del_0+6)+4\del+1=:C_1$. This implies that $\gam$ is contained in the $C_1$-neighborhood of $\beta$ with respect to $d$.

Now take $v \in \beta$, so that
$(x | y)_{v, d}  \leq 3\al/2$. By Lemma \ref{lem.RGforHDLF}, we can deduce that there exists $D\geq 0$ independent of $\beta$ and $\gamma$ and some $u\in \gam$ such that $$\max(|d(x,u)-d(x,v)|,|d(u,y)-d(v,y)|)\leq D.$$
As before, we can conclude that $d(u,v)\leq C_2$ for some uniform $C_2$, and hence $\al$ and $\beta$ are $C$-Hausdorff close with respect to $d$, for $C=\max(C_1,C_2)$. This proves our claim and the implication $(1) \Rightarrow (2)$.

Conversely, suppose $d$ satisfies (2) so that it is $\alpha$-roughly geodesic for some $\al\geq 0$. 
Let $x, y, w \in \G$, and let $p$ be a $\left(\k_{0}, d_{0}\right)$-quasi-center for $x, y, w$ with $\k_0$ depending only on $d_0$. We claim that $p$ is a $(\tilde{\k}, d)$-quasi-center for $x, y, w$, with $\tilde{\k}$ independent of $x, y, w$.

Let $\gamma_{1}, \gamma_{2}, \gamma_{3}$ be $\left(\alpha_{0}, d_{0}\right)$-rough geodesics joining $x$ and $y, y$ and $w$, and $w$ and $x$, respectively. Then there is some $D_{0}$ depending only on $\delta_{0}, \alpha_{0}$ and $\k_{0}$ such that $d_{0}\left(p, \gamma_{i}\right) \leq D_{0}$ for $i \in\{1,2,3\}$.
If $\beta_{1}, \beta_{2}, \beta_{3}$ are $(\alpha, d)$-rough geodesics joining $x$ and $y, y$ and $w$, and $w$ and $x$, respectively, then by (2), there is some $C \geq 0$ depending only on $d_0$ and $d$ such that $\beta_{i}$ and $\gamma_{i}$ are $C$-Hausdorff close in $(\G, d)$. 

Also, since $d_0$ is quasi-isometric to a word metric and $\G$ is finitely generated, we can find some $\lam_0>0$ such that $d\leq \lam_0d_0+\lam_0$. In particular we get $$d\left(p, \beta_{i}\right) \leq \lambda_{0} D_{0}+\lambda_{0}+C$$ for all $i \in\{1,2,3\}$, implying
$$\max \left\{(x | y)_{p, d},(y | w)_{p, d},(w | x)_{p, d}\right\} \leq 3 \alpha / 2+\lambda_{0} D_{0}+\lambda_{0}+C.$$
Therefore, $p$ is a $(\tilde{\k}, d)$-quasi-center, with $\tilde{\k}=3 \alpha / 2+\lambda_{0} D_{0}+\lambda_{0}$ $+C$, which proves the claim.

From this, we deduce
\begin{align*}
    (x | y)_{w, d} & \leq \tilde{\k}+d(p, w) 
 \leq \tilde{\k}+\lambda_{0} d_{0}(p, w)+\lambda_{0} \\
 & \leq \tilde{\k}+\lambda_{0}\left[\k_{0}+(x | y)_{w, d_{0}}\right]+\lambda_{0} \\
 & =\lambda_{0} (x | y)_{w, d_{0}}+\lambda_{0}+\lambda_{0} \k_{0}+\tilde{\k},
\end{align*}
and $d$ satisfies 1) with $\lambda=\lambda_{0}+\lambda_{0} \k_{0}+\tilde{\k}$.

To finish the proof we need to prove the furthermore statement. If $d$ and $d_0$ are as in the statement of the lemma and $(1)$ and $(2)$ hold, then we saw in the proof of $(2) \implies (1)$ that every $\left( \k_0,d_{0}\right)$-quasi-center for $x, y, z$ is also a $(\tilde{\k},d)$-quasi-center, with $\tilde{\k} > 0$ independent of $x,y,z$. This quasi-center will be uniformly close to any uniform rough geodesic connecting a pair of points among $x,y,z$. Hence $d$ is hyperbolic.
\end{proof}
 As a corollary we deduce the following aforementioned result.
\begin{corollary}\label{cor.hyp}
        Every pseudo metric $d \in \ov\Dc_\G$ is hyperbolic.
\end{corollary}

We also need a criterion that guarantees non-triviality of the stable translation length. 
\begin{lemma}\label{lem.unbound+trl} Let $d$ be a left-invariant, $\delta$-hyperbolic, and $\alpha$-roughly geodesic pseudo metric on the (non-necessarily hyperbolic) group $\Gamma$. Then $\ell_d$ is non-identically zero if and only if $(\G,d)$ is unbounded.
\end{lemma}
\begin{proof} It is enough to prove that if $\diam(\Gamma, d) \geq L:=9 \alpha+12 \delta+2$, then there is some $x \in \Gamma$ with $\ell_{d}[x]>0$. To this end, let $x \in \Gamma$ be such that $d(x, o) \geq L$. By our $\alpha$-rough geodesic assumption there is some $u \in \Gamma$ such that if we set $v:=u^{-1} x$, then
$$
|d(v, o)-d(x, o)/2| \leq (3 \alpha+1) / 2 \hspace{2mm}\text { and }\hspace{2mm} d(u, o)+d(v, o) \leq d(x, o)+3 \alpha.
$$
Also, by \cite[Thm.~1.2]{oregon-reyes.ineq} applied to $f=u, g=v$ and with base point the identity element $o \in \Gamma$, we get
$$
d(x, o) \leq \max \left\{d(u, o)+\ell_{d}[v], d(v, o)+\ell_{d}[u],
\frac{d(u, o)+d(v, o)+\ell_{d}[x]}{2}\right\}+6 \delta .
$$
Therefore, either some of the elements $x, u, v$ have positive stable translation lengths, or
$$
\begin{aligned}
L \leq d(x, o) & \leq \max \left\{d(u, o), d(v, o), \frac{d(u, o)+d(v, o)}{2}\right\}+6 \del \\
& \leq d(x, o)/2+(9 \alpha+1) / 2+6 \delta,
\end{aligned}
$$
which is a contradiction since $L>9 \alpha+1+12 \delta$.
\end{proof}

We also require the following lemma, which states that $\ov\calD_\G$ is closed under equivariant quasi-isometry among rough geodesic pseudo metrics. It is an immediate consequence of Proposition \ref{prop.qiGromovprod} and the invariance of hyperbolicity under quasi-isometry.

\begin{lemma}\label{lem.boundaryqi} Let $d \in \ov\calD_{\Gamma}$, and let $\tilde{d}$ be a roughly geodesic, left-invariant pseudo metric on $\Gamma$ such that the identity map  $id :(\Gamma, d) \rightarrow (\Gamma, \tilde{d})$ is a quasi-isometry. Then $\tilde{d} \in \ov\calD_{\Gamma}$, and $d\in \partial_M\calD_\G$ if and only if $\tilde{d}\in \partial_M\calD_\G$.
\end{lemma}


\subsection{Bounded backtracking and actions on $\R$-trees}
An $\mathbb{R}$-\emph{tree} is a metric space such that any two points can be joined by a unique embedded arc, whose length coincides with the distance of the given points. Equivalently, an $\mathbb{R}$-tree is a geodesic, $0$-hyperbolic metric space. Some hyperbolic groups act naturally and non-trivially on $\R$-trees, see for instance \cite{bestvina.feighn.stable}, \cite{paulin.trees}.

Extending the definition given in \cite{gaboriau.jaeger.levitt.lustig}, we say that the isometric action of the hyperbolic group $\G$ on the $\R$-tree $(T,d_T)$ has \emph{bounded backtracking} if the following holds: for some (any) finite, symmetric, generating subset $S \subset \G$ and some (any) $p \in T$, there exists $C \geq 0$ such that if $\gamma \subset \G$ is a geodesic in $d_S$ joining $o$ and $x$, then $\gamma \cdot p$ is $C$-Hausdorff close to the geodesic in $T$ joining $p$ and $xp$. 

The next proposition relates bounded backtracking and pseudo metrics belonging to $\ov\calD_\G$.

\begin{proposition}\label{prop.BBTtrees}
    Suppose $\G$ acts isometrically on the $\R$-tree $T$, so that the action has no global fixed point. Then the orbit pseudo metrics for the action of $\G$ on $T$ belong to $\ov\calD_\G$ if and only if the action has bounded backtracking. In particular, when $\G$ is not virtually free, isometric actions with bounded backtracking induce pseudo metrics belonging to $\partial_M \calD_\G$.
\end{proposition}

When $\G=F_n$ is a finitely generated non-abelian free group, Guirardel showed that every \emph{small}, minimal, isometric action of $\G$ on an $\mathbb{R}$-tree has bounded backtracking \cite[Cor.~2]{guirardel}. Therefore, Proposition \ref{prop.BBTtrees} implies item (3) of Theorem \ref{thm.examples} in the case of free groups. In addition, the Culler-Morgan compactification $\ov{\scrC \scrV} (F_n)$ of the Outer space coincides with the space of (rough similarity) classes  of orbit pseudo metrics induced by very small isometric actions of $\G$ on $\R$-trees \cite[Thm.~2.2]{bestvina.feighn.outer}. Therefore, from Proposition \ref{prop.BBTtrees} we deduce that $\ov{\scrC \scrV} (F_n)$ naturally injects into $\ov\scrD_{F_n}$. 

\begin{corollary}\label{cor.smallfreegroup}
Let $\G=F_n$ be a finitely generated non-abelian free group acting isometrically on the $\R$-tree $T$ so that the action is small. Then the orbit pseudo metrics induced by this action belong to $\ov\calD_{F_n}$. 
In particular, there exists a natural injective map $\ov{\scrC \scrV}(F_n) \hookrightarrow \ov\scrD_{F_n}$ that sends the Culler-Vogtmann boundary $\partial\scrC\scrV(F_n)$ into $\partial_M \scrD_{F_n}$.
\end{corollary}

For the proof of Proposition \ref{prop.BBTtrees}, we need a preliminary lemma.
\begin{lemma}\label{lem.roughgeodRtree}
    Let $\G$ be a (not necessarily hyperbolic) finitely generated group acting isometrically on the $\R$-tree $(T,d_T)$. Then for any $p\in T$, the pseudo metric $d_T^p(x,y)=d_T(xp,yp)$ on $\G$ is hyperbolic and roughly geodesic.
\end{lemma}
\begin{proof}
Clearly, $d$ is hyperbolic. To show it is roughly geodesic, let $S \subset \G$ be a finite, symmetric generating set, and let $\phi:\Cay(\G,S) \ra T$ be the unique $\G$-equivariant map such that $\phi(o)=p$, and each edge from $o$ to $s\in S$ in $\Cay(\G,S)$ is linearly mapped to the geodesic in $T$ joining $p$ and $sp$. Then $\phi$ is $L$-Lipschitz, with $L=\max_{s\in S}{d_T(p,sp)}$. 

Now, let $x,y\in \G$, and let $[x,y]_T$ denote the unique geodesic segment in $T$ joining $xp$ and $yp$. Since $\phi$ is continuous, for any geodesic path $\gam\subset \Cay(\G,S)$ joining $x$ and $y$, the image $\phi(\gam)$ contains $[x,y]_T$. Therefore, if $xp=p_0,p_1,\dots,p_n=yp$ is a 1-rough geodesic in $T$, then for any $i$ there is some $q_i\in \gam$ such that $d_T(p_i,\phi(q_i))\leq 3/2$. Also, for each $q_i$ there is some vertex $x_i\in \G$ such that $d_S(q_i,x_i)\leq 1$, and hence $d_T(p_i,x_ip)=d_T(p_i,\phi(x_i))\leq d_T(p_i,\phi(q_i))+d_T(\phi(q_i),\phi(x_i)) \leq 3/2+L/2$. If we choose $x_0=x$ and $x_n=y$, we conclude that the sequence $x=x_0,x_1,\dots,x_n=y$ is a $(4+L,d_T^p)$-rough geodesic joining $x,y \in \G$, which completes the proof.
\end{proof}

\begin{proof}[Proof of Proposition \ref{prop.BBTtrees}]
Let $\G$ act on the $\R$-tree $(T,d_T)$ as in the statement, and for $p\in T$, consider the pseudo metric $d_T^p$. This pseudo metric has non-constant stable translation length function since the action has no global fixed point  (by Lemma \ref{lem.unbound+trl}). As a consequence of Lemma \ref{lem.roughgeodRtree}, $d_T^p$ also satisfies property $i)$ of Lemma \ref{lem.BCCchar}. Therefore, the theorem follows by Lemma \ref{lem.BCCchar}, since $d_T^p$ satisfying property $ii)$ of that lemma is equivalent to the action having bounded backtracking.
\end{proof}


\subsection{Liouville embedding of the space of projective geodesic currents}
Throughout this section, let $\G$ be the fundamental group of a closed orientable surface of negative Euler characteristic, and fix a free and cocompact isometric action of $\G$ on the hyperbolic plane $(\Hy^2,d_{\Hy^2})$. Also, let 
$$i:\Curr(\G) \times \Curr(\G) \ra \R$$
be Bonahon's geometric intersection number \cite{bonahon.currentsTeich}, and for $g\in \G$, let $\eta_{[g]}\in \Curr(\G)$ be the rational geodesic current associated to $[g]$. The action of $\G$ on $\Hy^2$ induces a $\G$-equivariant bijection between the set $\calG$ of geodesics in $\Hy^2$ and $\partial^2 \G /\cor{\pm}$, where in $\partial ^2 \G$ we mod out by the involution $(p,q) \leftrightarrow (q,p)$. In this way, we consider geodesic currents as $\G$-invariant, locally finite measures on $\calG$.

Following \cite[Sec.~4]{burger-iozzi-parreau-pozzetti}, to each $\mu\in \Curr(\G)$ we construct the pseudo metric $d_\mu$ on $\Hy^2$ as follows: for $x,y\in \Hy^2$, let $[x,y]$ denote the closed geodesic interval in $\Hy^2$ joining $x$ and $y$, and we also define $(x,y]=[x,y] \bs \cor{x}$ and $[x,y)=[x,y] \bs \cor{y}$. If $I \subset \Hy^2$ is any subset, we let $\calG^\perp_I$ denote the set of geodesics in $\Hy^2$ intersecting $I$ exactly once. In this way, the pseudo metric $d_\mu$ is given by
$$d_\mu(x,y)=\frac{1}{2}(\mu(\calG^\perp_{[x,y)})+\mu(\calG^\perp_{(x,y]})).$$
In \cite[Prop.~4.1]{burger-iozzi-parreau-pozzetti} is was proven that $d_\mu$ is indeed a \emph{straight} pseudo metric, meaning that for $x,z\in \Hy^2$ and $z\in [x,y]$ it holds that
$$d_\mu(x,z)=d_\mu(x,y)+d_\mu(y,z).$$
This fact together with \cite[Lem.~4.7]{burger-iozzi-parreau-pozzetti} imply that
\begin{equation}\label{eq.intnumb}
\ell_{d_\mu}[g]=i(\mu,\eta_{[g]})
\end{equation}
for all $\mu\in \Curr(\G)$ and $g\in \G$.

From equation \eqref{eq.intnumb} and the work of Otal \cite{otal}, we see that two pseudo metrics $d_\mu$ and $d_{\mu'}$ on $\Hy^2$ are roughly similar through the identity map on $\Hy^2$ if and only if $\mu=\lam \mu'$ for some $\lam>0$. For non-zero geodesic currents, these pseudo metrics induce metric structures in $\ov\scrD_\G$.

\begin{proposition}\label{prop.embedcurrents}
Let $\G$ be a surface group acting geometrically on $\Hy^2$. Then for any non-zero geodesic current $\mu\in \Curr(\G)$, the orbit pseudo metrics induced by the action of $\G$ on $(\Hy^2,d_\mu)$ belong to $\ov\calD_\G$, and they belong to $\calD_\G$ if and only if $\mu$ is filling.
\end{proposition}

\begin{corollary}\label{coro.boundTeichembeds}
The assignment $\mu \mapsto \rho_{(\Hy^2,d_\mu)}$ induces an injective map from the space $\bbP\Curr(\G)$ of projective geodesic currents into $\ov\scrD_\G$. This map sends the Thurston boundary $\partial\scrT_\G$ into $\partial_M \scrD_{\G}$.
\end{corollary}

We begin the proof of Proposition \ref{prop.embedcurrents}, so we fix $\mu \in \Curr(\G)$.
\begin{lemma}\label{lem.muH2}
    There exists $\lambda_0>0$ such that
$$
d_\mu(x, y) \leq \lambda_0 d_{\Hy^2}(x, y)+\lambda_0
\hspace{2mm} \text{ for all } x, y \in \Hy^2.$$
\end{lemma}
\begin{proof}
    Given $A \geq 0$, we claim that there exists $B_A \geq 0$ such that $d_{\Hy^2}(x, y) \leq A$ implies $d_\mu(x, y) \leq B_A$. Indeed, since the action of $\Gamma$ on $\Hy^2$ is cocompact, there exists a compact set $K \subset \Hy^2$ such that if $d_{\Hy^2}(x, y) \leq A$, then $gx, gy \in K$ for some $g \in \Gamma$. The set $\calG_K \subset \calG$ of geodesics intersecting $K$ is compact, so that $B_A:= \mu\left(\calG_K\right)$ is finite. Therefore, if $x, y \in \Hy^2$ satisfy $d_{\Hy^2}(x, y) \leq A$ and $g$ is as above, we deduce that $d_\mu(x, y)=d_\mu(gx, gy) \leq \mu\left(\calG_K\right)=B_A$, which proves the claim.
    
Let $\lambda_0:=B_1$. If $x, y \in \Hy^2$ and $n=\lfloor d_{\Hy^2}(x, y) \rfloor$, let $x=x_0, x_1, \ldots, x_n \in [x, y]$ be such that $d_{\Hy^2}\left(x, x_i\right)=i$ for all $0 \leq i \leq n$. We get
$$
d_\mu(x, y)  \leq d_\mu\left(x_0, x_1\right)+\cdots+d_\mu\left(x_{n-1}, x_n\right)+d_\mu\left(x_n, y\right) 
 \leq(n+1) \lambda_0 \leq \lambda_0 d_{\Hy^2}(x, y)+\lambda_0,
$$
as desired.
\end{proof}
\begin{proof}[Proof of Proposition \ref{prop.embedcurrents}]
Let $\mu \in \Curr(\G)$ be non-zero, and let $\lambda_0$ be the constant from Lemma \ref{lem.muH2}. We claim that there exists $\lambda_1>0$ such that
\begin{equation}\label{eq.muHy}
 (x|y)_{w, d_\mu} \leq \lambda_1(x|y)_{w, d_{\Hy^2}}+\lambda_1   
\end{equation}
for all $x, y, w \in \Hy^2$. To this end, let $p$ be a $(\k, d_{\Hy^2})$-quasi-center for $x, y, w$, with $\k$ independent of this triple. If $q$ is the point in $[x, y]$ closest to $p$, then $d_{\Hy^2}(p, q) \leq (x|y)_{p, \Hy^2} \leq \k$, so that $d_\mu(p, q) \leq \lambda_0 \k+\lambda_0$.
Since $d_\mu$ is straight, we also have
$$
\begin{aligned}
d_\mu(x, p)+d_\mu(p, y) & \leq 2 \lambda_0 \k+2 \lambda_0+d_\mu(x, q)+d_\mu(q, y) \\
&=2 \lambda_0 \k+2 \lambda_0+d_\mu(x, y),
\end{aligned}
$$
and hence $(x | y)_{p, d_\mu} \leq \tilde\k:=\lambda_0 \k+\lambda_0$.
Similarly, we obtain $(x | w)_{p, d_\mu}, (w |y)_{p, d_\mu} \leq \tilde\k$, so that $p$ is a $\left(\tilde\k, d_\mu\right)$-quasi-center for $x, y$ and $w$. In particular, we deduce
$(x | y)_{w, d_\mu} \leq d_\mu(w, p)+\tilde{\k}$
$\leq \lambda_0(x |y)_{w,d_{\Hy^2}}+\lambda_0+\tilde{\k}+2 \lambda_0 \k$,
which proves the claim with $\lambda_1=\lambda_0+\tilde{\k}+2 \lambda_0 \k$.

Now take $w \in \mathbb{H}^2$  and let $d_\mu^w, d_{\Hy^2}^w$ be the corresponding orbit pseudo metrics induced by the action of $\G$ on $\Hy^2$. These pseudo metrics also satisfy a version of \eqref{eq.muHy}, and since $d_{\Hy^2}^w \in \calD_{\Gamma}$ we see that $d_\mu^w$ satisfies the inequality \eqref{eq.ineqdefboundarymetric} in Definition \ref{def.Manhattanboundarymetrics}. Also, since $\mu$ is non-zero, there exists $g \in \Gamma$ such that $i\left(\mu, \eta_{[g]}\right)=\ell_{d_\mu}[g]>0$, implying $d_\mu^w \in \ov\calD_{\Gamma}$.

Finally, if $L \in \Curr(\Gamma)$ is the Liouville current for the action of $\Gamma$ on $\Hy^2$, then $i(L, \beta)>0$ for all $\beta \in \Curr(\Gamma) \backslash\{0\}$, and hence the function $\Phi:\bbP\Curr(\Gamma) \rightarrow \R_{\geq 0}$ given by
$$
[\beta] \mapsto \frac{i(\mu, \beta)}{i(L, \beta)}
$$
is well-defined, continuous, and positive. Since $\bbP\Curr(\G)$ is compact and $i(L,\eta_{[g]})=\ell_{d_{\Hy^2}}[g]$ for all $g \in \Gamma$, we deduce that $\mu$ is filling if and only if there exists $A>0$ such that $\ell_{d_\mu}[g] \geq A \ell_{d_{\Hy^2}}[g]$ for all $g \in \Gamma$, which happens if and only if $d_\mu^w \in \calD_{\Gamma}$.
\end{proof}

By Skora's theorem \cite{skora}, the Thurston boundary of $\partial \scrT_\G$ can be described as the space of (rough similarity) classes of orbit pseudo metrics induced by small isometric actions of $\G$ on $\R$-trees. Therefore, Proposition \ref{prop.BBTtrees} and Corollary \ref{coro.boundTeichembeds} imply that if $\G$ is a surface group, then every small action of $\G$ on an $\R$-tree induces pseudo metrics in $\ov\calD_\G$. This proves item (3) of Theorem \ref{thm.examples} in the case of surface groups.


\subsection{Combinatorial examples} In this section we prove Propositions \ref{prop.exconeoff} and \ref{prop.exCAT(0)}, which correspond to items (1) and (2) of Theorem \ref{thm.examples}, respectively. We start with a connected graph $X$ with simplicial metric $d_X$, and let $\mathbb{K}=\cor{X_j}_{j\in J}$ be a family of subsets of vertices of $X$. From this data, we construct the connected graph $X_{\bbK}$ obtained by adding to $X$ the new edges $e_{x,y,j}$ with endpoints $x, y$ whenever $j \in J$ and $x, y$ are vertices of $X_j$ (thus $X$ is a subgraph of $X_\bbK$). Let $d_{X,\bbK}$ be the simplicial metric on $X_\bbK$. The following result is due to Kapovich and Rafi \cite[Prop.~2.6]{kap.rafi}, and will be used along with Lemma \ref{lem.BCCchar} to find examples of pseudo metrics in $\ov\calD_\G.$ 

\begin{proposition}\label{prop.conedoffgen}
    Let $X$ be a connected graph such that the simplicial metric $d_X$ is hyperbolic, and let $\bbK$ be a family of uniformly quasi-convex subsets of vertices of $X$. Then $(X_{\bbK},d_{X,\bbK})$ is also hyperbolic, and there is a constant $C\geq 0$ such that whenever $x,y \in X^{(0)}$, $[x,y]_X$ is a $d_X$-geodesic from $x$ to $y$ in $X$ and $[x,y]_{X,\bbK}$ is a $d_{X,\bbK}$-geodesic from $x$ to $y$ in $X_\bbK$, then $[x,y]_X$ and $[x,y]_{X,\bbK}$ are $C$-Hausdorff close in $(X_\bbK,d_{X,\bbK})$.
\end{proposition}

Now, let $S \subset \G$ be a finite, symmetric generating set with Cayley graph $(\Cay(\G,S),d_S)$. If $\calH$ is a set of subgroups of $\G$, the \emph{coned-off Cayley graph} $(\Cay(\G,S,\calH), d_{S,\calH})$ is defined as follows. Let $X = (\Cay(\G,S),d_S)$ and  take $\mathbb{K}$ to be the collection of all left cosets $xH$ for $H \in \mathcal{H}$ and $x\in \G$. Then $(\Cay(\G,S,\calH), d_{S,\calH})$ is defined to be the graph $(X_\mathbb{K}, d_{X,\mathbb{K}})$ introduced above.

When we cone-off finitely many quasi-convex subgroups, the orbit pseudo metrics induced by the action of $\G$ on the corresponding coned-off Cayley graphs will belong to $\ov\calD_\G$.

\begin{proposition}\label{prop.exconeoff} Let $\calH$ be a finite set of quasi-convex subgroups of $\G$, and for a finite, symmetric generating set $S \subset \G$, consider the coned-off Cayley graph $\Cay(\G,S,\calH)$. If all the subgroups in $\calH$ are infinite index in $\G$, then the orbit pseudo metrics induced by the action of $\G$ on $\Cay(\G,S,\calH)$ belong to $\ov\calD_\G$. In addition, these pseudo metrics belong to $\partial_M\calD_\G$ if and only if some subgroup in $\calH$ is infinite.
\end{proposition}

\begin{proof} 
We apply Proposition \ref{prop.conedoffgen} to $X=\Cay(\Gamma, S)$ and $\bbK=\{xH \colon x \in \Gamma, H \in \mathcal{H}\}$, so that the inclusion $\Cay(\Gamma, S) \rightarrow X_{\bbK}=\Cay(\G,S,\calH)$ maps geodesics in $\Cay(\Gamma, S)$ uniformly close to geodesics in $\Cay(\Gamma, S,\calH)$. Since $\Cay(\G,S,\calH)$ is geodesic by construction, any orbit pseudo metric from the isometric action of $\Gamma$ on $(\Cay(\G,S,\calH),d_{S,\calH})$ will be roughly geodesic, and hence will satisfy properties $i)$ and $ii)$ of Lemma \ref{lem.BCCchar}. By that lemma, to conclude that orbit pseudo metrics induced by $\Cay(\G,S,\calH)$ belong to $\ov\calD_\G$ we are left to show that the action of $\G$ on $\Cay(\G,S,\calH)$ has at least one loxodromic element, which we now explain. By \cite[Cor.~6.15]{abbott.manning}, a non-torsion element $x\in \G$ acts loxodromically on $\Cay(\G,S,\calH)$ if and only if no power of $x$ lies in a conjugate of some subgroup in $\calH$. 
That corollary applies when $\calH$ consists of a single quasi-convex subgroup, but the same argument works in the general case if we replace \cite[Thm.~3.2]{abbott.manning} by the fact that any finite collection of quasi-convex subgroups has \emph{finite height} (see e.g.~\cite[Main Theorem]{GMRS}). Since we are assuming that each member of $\calH$ is infinite index in $\G$, such an element $x$ exists, and an argument for this is as follows. If we fix a visual metric on $\partial \G$ and $\ep>0$ small enough, by \cite[Cor.~2.5]{GMRS} the set $\calH_\ep$ of conjugates of groups in $\calH$ with limit set having diameter at least $\ep$ is finite. Also, the pairs of fixed points of loxodromic elements are dense in $\partial^2\G$, so by our infinite index assumption we can find a loxodromic $x$ with fixed points in $\partial \G$ at distance at least $\ep$ but not contained in any limit set of a group in $\calH_\ep$.  
Finally, it is clear that these pseudo metrics belong to $\calD_\G$ if and only if all the subgroups in $\calH$ are finite.
\end{proof}

A $\CAT(0)$ \emph{cube complex} is a simply connected, metric polyhedral complex in which all polyhedra are unit-length Euclidean cubes, and satisfies Gromov's link condition: the link of each vertex is a flag complex. For references about the geometry of $\CAT(0)$ cube complexes see \cite{bridson.haefliger, sageev}.

We can apply the proposition above to show that cocompact actions on $\CAT(0)$ cube complexes induce pseudo metrics in $\ov\calD_\G$, as long as the hyperplane stabilizers are quasi-convex.
\begin{proposition}\label{prop.exCAT(0)}
Let $(\calX,d_\calX)$ be a $\CAT(0)$ cube complex with combinatorial metric $d_\calX$, and assume $\G$ acts cocompactly on $\calX$ by simplicial isometries. Also, suppose that:
\begin{enumerate}
\item hyperplane stabilizers are quasi-convex; and
\item the action has no global fixed point.
\end{enumerate}
Then the orbit pseudo metrics for the action of $\G$ on $\calX$ belong to $\ov\calD_\G$. In addition, they belong to $\partial_M\calD_\G$ if and only if some vertex stabilizer is infinite.
\end{proposition}

Specializing the proposition above to 1-dimensional $\CAT(0)$ cube complexes, we obtain that Bass-Serre tree actions with quasi-convex edge groups induce pseudo metrics in $\ov\calD_\G$.

\begin{corollary}\label{coro.exBStree}
    Let $T$ be a Bass-Serre tree for a finite graphs of groups decomposition of $\G$. Suppose this action satisfies:
\begin{enumerate}
\item the edge subgroups are quasi-convex in $\G$; and
\item the vertex subgroups are infinite index in $\G$.
\end{enumerate}
Then the orbit pseudo metrics for the action of $\G$ on $T$ belong to $\ov\calD_\G$. In addition, they belong to $\partial_M\calD_\G$ if and only if some vertex stabilizer is infinite.
\end{corollary}

\begin{proof}[Proof of Proposition \ref{prop.exCAT(0)}]
Let $\calH$ be a complete set of representatives of the conjugacy classes of vertex stabilizers for the action of $\Gamma$ on $\left(\calX, d_{\calX}\right)$. This set is finite, and since hyperplane stabilizers are quasi-convex, by \cite[Thm.~A]{groves.manning} all the subgroups in $\mathcal{H}$ are quasi-convex. Also, since the action of $\Gamma$ on $(\calX, d_\calX)$ is cocompact and has no global fixed point, it has unbounded orbits with respect to the $\CAT(0)$ metric on $\calX$, which is quasi-isometric to $d_\calX$. This implies that the action on $(\calX,d_\calX)$ has unbounded orbits, so all the subgroups in $\mathcal{H}$ are infinite index in $\Gamma$. Therefore, Proposition \ref{prop.exconeoff} applies to $\mathcal{H}$, and hence the orbit pseudo metrics induced by the action of $\Gamma$ on the coned-off Cayley graph $\Cay(\G, S, \calH)$ belong to $\ov\calD_{\G}$.

To conclude the result, by \cite[Thm.~5.1]{charney.crisp}, $\left(\calX, d_{\calX}\right)$ is $\Gamma$-equivariantly quasi-isometric to $\Cay(\G,S,\calH)$, and the first conclusion follows from Lemma \ref{lem.boundaryqi}. 
The second conclusion follows from the cocompactness of the action, since in this case, properness is equivalent to finiteness of all the vertex stabilizers.
\end{proof}

\begin{remark}So far, most of the examples
of pseudo metrics $d$ in $\partial_M\calD_\G$ that we have exhibited satisfy $\ell_d[x] = 0$ for some non-torsion element $x$. By contrast, in \cite{kap.loxo}, Kapovich constructed an example of a hyperbolic graph $(Y, d_Y )$ and an isometric action of the
rank-2 free group $F_2$ on $Y$ satisfying  $\ell_Y[x] \geq 1/7$ for any non-trivial $x\in F_2$. In addition, he proved that orbit pseudo metrics induced by this action are not quasi-isometric to a word metric, but that they satisfy conditions $i)$ and $ii)$ of Lemma \ref{lem.BCCchar}. Consequently the action on $Y$ induces a purely loxodromic boundary metric structure in $\partial_M\scrD_{F_2}$.
\end{remark}


\section{Extension of stable translation length to $\calC urr(\G)$ and a conjecture of Bonahon}\label{sec.extensiontranslength}

In this section, we study the extension of the stable translation length functions to the space of geodesic currents. First, we prove Theorem \ref{prop.contextbound}.

\begin{proof}[Proof of Theorem \ref{prop.contextbound}]
    Given $d_0\in \calD_\G$, by Corollary \ref{coro.dinftisdifference} there exists $d_1\in \calD_\G$ such that 
    $$\ell_d=\Dil(d_0,d_1)\ell_{d_1}-\ell_{d_0}.$$
    If $\G$ is torsion-free, by \cite[Cor.~5.2]{oregon-reyes.ms}, the stable translation lengths $\ell_{d_0}$ and $\ell_{d_1}$ can be extended continuously to $\Curr(\G)$, so the same holds for $\ell_d$. 
    In the general case that $\G$ contains the torsion-free group $\G_0$ as a finite index subgroup, the conclusion follows since every geodesic current on $\G$ is a geodesic current on $\G_0$.
\end{proof}

In 1988, Bonahon conjectured that the only isometric actions of a hyperbolic group $\G$ on real trees whose stable translation length continuously extends to $\Curr(\G)$ are those with virtually cyclic interval/edge stabilizers \cite[p.~164]{bonahon.currentshypgroups}. However, according to Corollary \ref{coro.exBStree} and Proposition \ref{prop.contextbound}, such a continuous extension exists for every Bass-Serre tree action with quasi-convex edge subgroups. As we can produce examples of splittings over quasi-convex, non-virtually cyclic subgroups, we deduce Theorem \ref{Bonahoncounterex} from the introduction which disproves Bonahon's conjecture.

\begin{example}\label{ex.qcsurfacegroupsplitting}
If $M_0$ is any closed, hyperbolic 3-manifold, there exists a finite cover $M$ of $M_0$ and an embedded, incompressible connected, 2-sided closed surface $S\subset M_0$ such that $H=\pi_1(S)$ is quasi-convex in $\G=\pi_1(M)$  (this follows from the virtual Haken theorem \cite{agol}). Cutting $M$ along $S$ gives a splitting of $\G$ over $H$, and the stable translation length of the Bass-Serre tree corresponding to this splitting extends continuously to $\Curr(\G)$ by quasi-convexity of $H$. The action of $\G$ on this tree is not small.
\end{example}

\begin{example}\label{ex.cubulated}
Generalizing the example above, let $\G$ be any hyperbolic group acting properly and cocompactly by simplicial isometries on the $\CAT(0)$ cube complex $\calX$. Suppose also that there exists a non-virtually cyclic hyperplane stabilizer $H<\G$, which is infinite index in $\G$ and stabilizes the hyperplane $\calH$. By Agol's theorem \cite{agol}, there exists a finite index subgroup $\G_0<\G$ such that if $H_0:=H\cap \G_0$, then $\G_0 \bs \calX$ is a non-positively curved cube complex and $H_0\bs \calH$ is an embedded, two-sided hyperplane of $\G_0 \bs \calX$ that does not self-osculates. This implies that $\G_0$ splits over $H_0$, and since $H_0$ is non-virtually cyclic, the action of $\G_0$ on the corresponding Bass-Serre tree gives another counterexample to Bonahon's conjecture.
\end{example}


\section{Some questions}

In this final section we list some open questions related to $\ov\Dc_\G$. Our first two questions were mentioned in the introduction and are concerning which actions give rise to points in $\ov\Dc_\G$.

\begin{question}\label{question.acylindrical}
    If $\G$ is non-elementary hyperbolic, does any cobounded acylindrical action of $\G$ on a geodesic hyperbolic space induce a point in $\ov\scrD_\G$?
\end{question}

\begin{question}\label{qstn.Rtrees}
    Suppose $\G$ is non-elementary hyperbolic and acts minimally on an $\R$-tree with infinite index quasi-convex interval stabilizers. Does this action induce a point in $\ov\scrD_\G$?
\end{question}

\begin{question}
    Given two finite, symmetric, admissible probability measures $\mu,\mu_\ast$ on $\G$ their associated Green metrics $d_{\mu}$ and $d_{\mu_\ast}$ belong to $\calD_\G$ \cite[Cor.~1.2]{bhm}. We can find two points at infinity $\rho_{\infty}$ and $\rho_{-\infty}$ in the Manhattan boundary corresponding to the endpoints of the bi-infinite Manhattan geodesic joining $[d_{\mu}]$ and $[d_{\mu_\ast}]$. Is there any meaning for $\rho_\infty$ and $\rho_{-\infty}$ in terms of $\mu$ and $\mu_\ast$?
\end{question}

The next question follows naturally from the work of Cantrell and Tanaka \cite{cantrell.tanaka.Man}. 
\begin{question}
Are the Manhattan curves $\theta_{d_\ast/d}$ for $d,d_\ast \in \Dc_\G$ real analytic?
\end{question}

We can also ask for generalizations of $\ov\calD_\G$ and $\ov\scrD_\G$ to groups that are not necessarily hyperbolic.

\begin{question}
Is there a nice theory of metric structures for relatively hyperbolic groups? In this case, we could consider cusp uniform actions on roughly geodesic hyperbolic spaces.  
\end{question}

For an arbitrary finitely generated group $\G$ we can consider the space $\calD_\G$ of left-invariant, roughly geodesic pseudo metrics on $\G$ that are quasi-isometric to a word metric for a finite generating set. This is consistent with our definition of $\calD_\G$ when $\G$ is hyperbolic. In this more general setting, the expression $\Del$ in formula \eqref{eq.defDel} only defines a pseudo metric on the quotient $\scrD_\G$ of $\calD_\G$ under rough similarity. 

\begin{question}
    Is there a non-virtually cyclic group $\G$ for which the pseudo metric space $(\scrD_\G,\Del)$ has zero diameter? When non-zero, can it have finite diameter?  
\end{question}

\begin{question}
    Let $\G$ be such that $(\scrD_\G,\Del)$ is a metric space. Is this space connected? Is it contractible? 
\end{question}


\subsection*{Open access statement}
For the purpose of open access, the author has applied a Creative Commons Attribution (CC BY) licence to any Author Accepted Manuscript version arising from this submission.


\noindent\small{Department of Mathematics, 
University of Warwick,
Coventry, CV4 7AL, UK}\\
\small{\textit{Email address}: \texttt{stephen.cantrell@warwick.ac.uk}\\
\\
\small{Max Planck Institute for Mathematics, Bonn, Germany, 53111}\\
\small{\textit{Email address}: \texttt{eoregon@mpim-bonn.mpg.de}}\\

\end{document}